\documentclass[leqno,fleqn,12pt]{elsarticle} 

\usepackage{ifthen}
\newboolean{showhidden}
\setboolean{showhidden}{false}

\usepackage[margin=0.7in]{geometry}
\usepackage{amsthm}
\usepackage{amsmath}
\usepackage{amssymb}
\usepackage{algorithm}
\usepackage{url}
\usepackage{hyperref}
\usepackage{booktabs}
\usepackage{graphicx}

\newtheorem{theorem}{Theorem}

\newtheorem{example}[theorem]{Example}

\newtheorem{proposition}[theorem]{Proposition}
\newtheorem{corollary}[theorem]{Corollary}

\newcommand{\be}{\mathbf e}
\newcommand{\bu}{\mathbf u}
\newcommand{\bv}{\mathbf v}

\newcommand{\bxbar}{\overline\bx}
\newcommand{\bx}{\mathbf x}
\newcommand{\bX}{\mathbf X}
\newcommand{\bmu}{\boldsymbol{\mu}}

\newcommand{\calk}{\mathcal K}
\newcommand{\calo}{\mathcal O}
\newcommand{\calu}{\mathcal U}
\newcommand{\calv}{\mathcal V}

\newcommand{\bsize}{{\sf m}}
\newcommand{\kmin}{{\sf m_1}}
\newcommand{\kmax}{{\sf m_2}}

\newcommand{\tr}{\mathop{{\rm tr}}}
\newcommand{\R}{\mathbb R}
\newcommand{\C}{\mathbb C}

\newcommand{\wh}{\widehat}
\newcommand{\wt}{\widetilde}

\newcommand{\diag}{{\rm diag}}
\newcommand{\rank}{{\rm rank}}
\DeclareMathOperator*{\argmin}{\arg\!\min}

\newcommand{\expec}{\mathbb E}
\newcommand{\var}{\text{Var}}
\newcommand{\ph}{\phantom}

\newcommand{\smtxa}[2]{
{\mbox{\scriptsize
$\left[\!\! \begin{array}{#1} #2 \end{array} \!\! \right]$}}}

\newcommand{\trks}{{\sf TR KSchur}}
\newcommand{\trsub}{{\sf TR subspace}}
\newcommand{\fdasub}{{\sf FDA subspace}}
\newcommand{\Vopt}{V^\ast}
\newcommand{\rhopt}{\rho^\ast}

\usepackage{xcolor}

\usepackage[normalem]{ulem}

\title{A subspace method for large-scale trace ratio problems}
\journal{-} 

\begin{document}

\begin{frontmatter} 

\author[tue]{Giulia Ferrandi}
\author[tue]{Michiel E.~Hochstenbach}
\author[IST]{M.~Ros\'ario Oliveira}

\address[tue]{Department of Mathematics and Computer Science, TU Eindhoven, PO Box 513, 5600MB Eindhoven, The Netherlands, {\tt \{g.ferrandi, m.e.hochstenbach\}@tue.nl}}
\address[IST]{CEMAT and Department of Mathematics, Instituto Superior T\'ecnico, Universidade de Lisboa, Portugal, 
{\tt rosario.oliveira@tecnico.ulisboa.pt}}

\begin{abstract}
A subspace method is introduced to solve large-scale trace ratio problems. This approach is matrix-free, requiring only the action of the two matrices involved in the trace ratio. At each iteration, a smaller trace ratio problem is addressed in the search subspace. Additionally, the algorithm is endowed with a restarting strategy, that ensures the monotonicity of the trace ratio value throughout the iterations. The behavior of the approximate solution is investigated from a theoretical viewpoint, extending existing results on Ritz values and vectors, as the angle between the search subspace and the exact solution approaches zero. Numerical experiments in multigroup classification show that this new subspace method tends to be more efficient than iterative approaches relying on (partial) eigenvalue decompositions at each step.
\end{abstract}

\begin{keyword}
Trace ratio, 
subspace method, 
Davidson's method, 
linear dimensionality reduction, 
Fisher's discriminant analysis, 
multigroup classification. 
\end{keyword}

\end{frontmatter} 

\section{Introduction}
Let $A$ and $B$ be $p\times p$ matrices, where $A$ is symmetric and $B$ is symmetric positive definite (SPD).
Let $V$ be $p \times k$, with orthonormal columns, and $I_k$ be the $k \times k$ identity matrix.
We develop and investigate a subspace method for the trace ratio optimization problem (TR, cf.~\cite{ngo2012trace})
\begin{equation} \label{eq:TR.MaxProblem}
\max_{V^TV = I_k} \rho(V) := \frac{\tr(V^T\!AV)}{\tr(V^T\!BV)},
\end{equation}
(a \emph{ratio of traces}) where $\tr(\cdot)$ indicates the trace of a matrix. The TR problem has often been studied as an alternative method to Fisher's discriminant analysis (FDA; see, e.g., \cite[Ch.~11]{johnson2007applied}) for dimensionality reduction and classification of multigroup data (cf.~Section~\ref{sec:classification}). For FDA, the problem of interest is the maximization of the {\em trace of a ratio}, which has the equivalent forms (cf.~\cite[Prop.~2.1.1]{absil2009optimization}):
\begin{equation}\label{eq:gep}
\max_{V^T\!BV = I_k} \tr(V^T\!AV) = \max_{\rank(V) = k} \tr((V^T\!BV)^{-1} \, V^T\!AV),
\end{equation}
where $\rank(\cdot)$ indicates the rank of a matrix. The maximizer of this problem is given by the $k$ generalized eigenvectors of the pair $(A,B)$, corresponding to the $k$ largest eigenvalues (see, e.g., \cite{ngo2012trace}). Therefore, finding the solution to FDA is equivalent to solving the generalized eigenvalue problem (GEP) $A\bx = \lambda B\bx$, for its $k$ largest eigenvalues. 

In general, TR does not have a closed-form solution. A typical Newton type iterative method to solve \eqref{eq:TR.MaxProblem} consists of solving a sequence of eigenvalue problems, where, at each iteration, the $k$ leading eigenvectors are required. In a context where $p$ is large, subspace methods are preferable to (dense) methods solving all eigenpairs. In \cite{zhang2010fast, ngo2012trace}, an implicitly restarted Lanczos method is employed for this task. Regarding the GEP (and also FDA), there exist several methods to retrieve the $k$ leading generalized eigenvectors from $(A,B)$, where $A$ is symmetric and $B$ is SPD; for instance, Davidson type methods (see, e.g., \cite{crouzeix1994davidson} or \cite[Sec.~8.3]{saad2011numerical}), LOBPCG \cite{knyazev2001toward}, or Jacobi--Davidson approaches for hard cases with clustered eigenvalues \cite{sleijpen2000jacobi}.

Inspired by Davidson type methods, we propose a subspace method for solving the TR problem. The idea is to produce a suitable search subspace by solving a sequence of trace ratio problems after projecting $A$ and $B$. A residual matrix is then used to enrich the search subspace, and a suitable restarting technique is exploited to lower both the computational cost and the memory requirements. This method is especially convenient if $A$ and $B$ are sparse, or the matrix-vector multiplications by $A$ and $B$ are cheap to compute. In the context of multigroup classification, both $A$ and $B$ can be factored conveniently (see, e.g., \cite{zhang2010fast}), making TR and FDA ideal candidates for subspace methods. 

The paper is organized as follows. Section~\ref{sec:tr-overview} gives an overview of the trace ratio problem, its properties, and the commonly used iterative method to find its solution. Section~\ref{sec:tr-davidson} describes the three phases of a new Davidson type method to solve TR: extraction, expansion, and restart. We show that the iterations of our algorithm generate a non-decreasing sequence of trace ratio values.
Section~\ref{sec:trsub-analysis} discusses the behavior of the approximate solution generated by the subspace method, and the related approximate eigenvalues, under the hypothesis that the angle between the search subspace and the exact solution approaches zero. In Section~\ref{sec:classification}, TR and FDA are reviewed in the context of multigroup classification. A Davidson's method for solving FDA is proposed in Section~\ref{sec:trsub-gep}.
This is already known in the context of subspace methods for eigenvalue problems, but we provide some further insight into the restarting mechanism. The two TR algorithms and Davidson's method for FDA are applied to multigroup classification in Section~\ref{sec:trsub-exp}, and tested on both synthetic and real datasets. Conclusions are drawn in Section~\ref{sec:trsub-concl}.

The eigenvalues of a symmetric matrix $A$ are denoted by $\lambda_1(A) \ge \cdots \ge \lambda_p(A)$.
Similarly, $\lambda_1(A,B) \ge \cdots \ge \lambda_p(A,B)$ are the eigenvalues of the pair $(A,B)$, where $A$ is symmetric and $B$ is SPD. The Frobenius norm of a matrix is indicated by $\|\cdot\|_F$; the spectral norm of a matrix is $\|\cdot\|$; $\|\bx\|_{A}^2 = \bx^T\!A\bx$ is used for the $A$-weighted norm associated with an SPD matrix $A$. 

\section{Overview of the trace ratio problem}\label{sec:tr-overview}
This section reviews some properties of the trace ratio problem and reports the well-known iterative method to find its solution. The trace ratio problem is well defined, under the appropriate hypotheses on $(A,B)$, since it admits a finite solution \cite[Prop.~3.2]{ngo2012trace}:
\begin{proposition}\label{prop:rank}
Let $A$ be symmetric and $B$ be symmetric positive semidefinite, with $\rank(B) \ge p-k+1$. Then \eqref{eq:TR.MaxProblem} admits a finite maximum.
\end{proposition}

Let $\rhopt$ be the global maximum of \eqref{eq:TR.MaxProblem}. Then, a necessary condition for $\Vopt$ to be a corresponding maximizer is (cf. \cite[Eq.~(4.1)]{ngo2012trace})
\begin{equation}\label{eq:tr-linearization}
\max_{V^TV=I_k}\tr(V^T(A-\rho^\ast B)\,V) = \tr((\Vopt)^T(A-\rho^\ast B)\,\Vopt) = 0.
\end{equation}
From the Courant--Fischer characterization (see, e.g., \cite{ngo2012trace}), $\Vopt$ is an orthonormal basis for the eigenspace of $A-\rho^\ast B$, associated with its $k$ largest eigenvalues.

A direct consequence of \eqref{eq:tr-linearization} is formulated in \cite[Thm.~2]{shen2010geometric}, which states that every local maximizer is also a global one. Moreover, \cite[Thm.~3.4]{zhang2010fast} shows that the condition \eqref{eq:tr-linearization} is also sufficient. The result is reported in the following proposition.
\begin{proposition}\label{prop:char}
Let $\rho^\ast$ be the global maximum of \eqref{eq:TR.MaxProblem}. Then any $\Vopt$ is a global maximizer for \eqref{eq:TR.MaxProblem} if and only if it is an orthonormal basis for the eigenspace of the matrix $A - \rho^\ast B$, corresponding to its $k$ largest eigenvalues. In addition, $\tr((\Vopt)^T(A - \rho^\ast B)\Vopt) = 0$.
\end{proposition}

Uniqueness (up to orthonormal transformations) is also discussed in, e.g., \cite{shen2010geometric,zhang2010fast,ngo2012trace}. From \cite[Thm.~3.4]{zhang2010fast}, the solution to \eqref{eq:TR.MaxProblem} is spanned by the $k$ leading eigenvectors of $A - \rho^\ast B$. This subspace is unique if there is a gap between the $k$th and the $(k+1)$st largest eigenvalues, i.e., $\lambda_k(A-\rho^\ast B) > \lambda_{k+1}(A-\rho^\ast B)$. In this situation, the solution is uniquely described by the set $\{VQ\, : \, Q \in \mathbb{R}^{k\times k},\, Q^T Q = I_k\}$. When there is no gap, the uniqueness of $\rho^\ast$ still holds, as a consequence of Proposition~\ref{prop:char}. Let us consider a small example.
\begin{example}
Let $A = \diag(3,2,1)$, $B = \diag(1,4,3)$, $k=2$. One may check that the maximum trace ratio is $\rhopt = 1$, with $A -\rhopt B = \diag(2,-2,-2)$. The maximizing subspace is spanned by the first canonical basis vector $\be_1$, together with any linear combination of $\be_2$ and $\be_3$.
\end{example}

Starting from \eqref{eq:tr-linearization}, the trace ratio problem \eqref{eq:TR.MaxProblem} can be recast as a root-finding problem for $f(\rho) = \max_{V^TV=I_k}\tr(V^T(A-\rho B)V)$ \cite{ngo2012trace}.
A common iterative method to find the zero of $f$ is given in Algorithm~\ref{algo:trsub-traceratio}.
The method may be viewed as member of the family of \emph{self-consistent field} methods; see, e.g., \cite{zhang2015maximization}, where a generalization of the trace ratio function is studied.
It also turns out to be a Newton type method, as pointed out in \cite{ngo2012trace}.

\begin{algorithm}
\caption{Iterative method to solve the trace ratio problem \eqref{eq:TR.MaxProblem}
}
\label{algo:trsub-traceratio}
{\bf Input:} Symmetric $A \in \R^{p \times p}$, SPD $B \in \R^{p \times p}$, dimension $k < p$; (optional) initial trace ratio value $\rho_0$, tolerance {\sf tol}. \\
{\bf Output:} $(p \times k)$ maximizer $V$ with maximum $\rho$. \\
\begin{tabular}{ll}
{\footnotesize 1:} & {\bf If} $\rho_0$ is not provided:\\
{\footnotesize 2:} & \phantom{M} Select random $n \times k$ matrix $V$, orthogonalize columns 
\\
{\footnotesize 3:} & \phantom{M} Determine $\rho_0 = \frac{\tr(V^T\!AV)}{\tr(V^T\!BV)}$ \\
{\footnotesize 4:} & {\bf end}\\
{\footnotesize 5:} & {\bf for} $i = 1, 2, \dots$ \\
{\footnotesize 6:} & \phantom{M} Determine $V$ as largest $k$ eigenvectors of $A-\rho_{i-1} B$ \\
{\footnotesize 7:} & \phantom{M} Update $\rho_i = \frac{\tr(V^T\!AV)}{\tr(V^T\!BV)}$\\
{\footnotesize 8:} & \phantom{M} {\bf if} converged (e.g., $1 - \rho_{i-1}/\rho_{i} < \sf{tol}$), {\bf return} $V,\, \rho_i$\\
{\footnotesize 9:} & {\bf end}
\end{tabular} 
\end{algorithm}

Algorithm~\ref{algo:trsub-traceratio} has been proposed in this form by several authors.
Guo et al.~\cite{guo2003generalized} introduce an iterative approach for solving \eqref{eq:TR.MaxProblem}, employing a bisection technique to update $\rho$, but still addressing an eigenvalue problem of the form $A-\rho B$ during the outer iterations. In the study by Wang et al.~\cite{Wang.et.al:2007}, the trace ratio algorithm already appears in the form of Algorithm~\ref{algo:trsub-traceratio}, but no particular emphasis is put on the problems arising from the large-scale setting. To the best of our knowledge, \cite{zhang2010fast} is the first to have proposed subspace methods for computing the required $k$ eigenvectors at each iteration. This work was soon followed by \cite{ngo2012trace}, which is a comprehensive review of the trace ratio method. In contrast with the formulation of other authors, Algorithm~\ref{algo:trsub-traceratio} starts from a known estimate of $\rho$, if available. This choice is convenient for the subspace method introduced in the next section. 

The global convergence of Algorithm~\ref{algo:trsub-traceratio} is stated in  \cite[Thm.~5.1]{zhang2010fast}, where the authors also provide a bound for the (linear) convergence rate:
\[\vert\rho_{i} - \rhopt\vert \le \Big(1 - \frac{\sum_{j=1}^{k}\lambda_{p-j+1}(B)}{\sum_{j=1}^{k}\lambda_{j}(B)}\Big) \, \cdot \vert \, \rho_{i-1} - \rhopt\vert.\]
This bound suggests that if some eigenvalues of $B$ are small (compared to the largest ones), the convergence of the method may be slower. 

Other convergence results are presented in \cite{zhang2010fast,zhang2013superlinear,cai2018eigenvector}. The local quadratic convergence of the sequences $\{V_i\}_i$ and $\{\rho_i\}_i$ is presented in \cite{zhang2010fast}, under the assumption that there is a gap between the $k$th and $(k+1)$st eigenvalues of $A-\rho^\ast B$. The local quadratic convergence of $\{V_i\}_i$ is also derived in \cite[Thm.~5.7]{cai2018eigenvector} as a consequence of more general results on the self-consistent field iteration. Superlinear convergence of the trace ratio values $\{\rho_i\}_i$ is proved in \cite[Thm.~3.1]{zhang2013superlinear} under no restrictions.

The most expensive part of Algorithm~\ref{algo:trsub-traceratio} is Line~6, especially in the large-scale setting, due to the need to solve an eigenvalue problem at each iteration. As proposed in \cite{zhang2010fast,ngo2012trace}, a first improvement is to compute the $k$ leading eigenvectors of $A-\rho_i B$ using subspace methods. Both papers employ the standard Matlab routine {\sf eigs} for this purpose. This is based on the implicitly restarted Arnoldi method (IRA; see, e.g., \cite[Sec.~5.2]{Ste01}). Refined Ritz vectors are exploited in \cite{jia2021convergence}. Our experiments make use of the thick restarted Lanczos method as implemented in \cite{wu2000thick}, which is a specific instance of the more general Krylov--Schur method \cite{stewart2002krylov}. While IRA and Krylov--Schur are mathematically equivalent when shifts are equal to standard Ritz values, the latter method is much easier to implement and more reliable than using implicit QR decomposition (cf. \cite[p.~337]{Ste01}). 

Existing state-of-the-art trace ratio algorithms for large-scale data have as common aspect that for every iterate $i$, a sequence of projected eigenvalue problems is solved, involving the (large) matrix $A - \rho_i B$. This means that a subspace method routine is called at each iteration. 
As an alternative strategy, we propose a subspace method which addresses a sequence of {\em projected trace ratio problems}. From this, an approximate solution $(V,\rho_i)$ is delivered at each outer iteration, and used to enrich the search subspace. In contrast with Algorithm~\ref{algo:trsub-traceratio}, this method does not solve any of the full eigenvalue problems for $A - \rho_i B$. 

\section{A Davidson type method for trace ratio problems}\label{sec:tr-davidson}
We develop a new Davidson type method for solving the trace ratio problem. This can be used either to find the maximizer of \eqref{eq:TR.MaxProblem} or to cheaply generate a quality approximation to $\rhopt$. 
Davidson's method dates back to 1975, and is mostly applied to eigenvalue problems (see, e.g., \cite[Sec.~8.3]{saad2011numerical} and \cite{crouzeix1994davidson,stathopoulos1998dynamic}) and generalized eigenvalue problems (see, e.g., \cite{morgan1990davidson}). 
The algorithm generally consists of three phases: during {\em extraction}, a projected eigenvalue problem is solved in the current search space; one or more preconditioned residual vectors (also called {\em corrections}) are used in the {\em expansion} phase, to improve the quality of the search subspace; an optional (but usually practically necessary) phase is {\em restarting}, where the current search subspace is shrunk to a smaller one, which ideally collects the most useful information gained in the previous iterations of the method. 

Our method includes all three phases, making some necessary adaptations to account for the changing eigenvalue problem during the outer iterations. The extraction phase is described in the following section.

\subsection{Subspace extraction}\label{sec:trsub-extra}
In our subspace method for TR, the subspace extraction consists of finding an approximate solution to the trace ratio problem \eqref{eq:TR.MaxProblem} in the $j$-dimensional search subspace $\calu_j$, where $k\le j\le p$. 
Suppose $U_j \in \mathbb R^{p\times j}$ is a matrix with orthonormal columns, spanning $\calu_j$. The desired approximate solution is then $V = U_jZ$, where $Z \in \mathbb R^{j\times k}$ solves the projected trace ratio problem: 
\begin{equation}\label{eq:proj-tr}
\rho_j := \max_{\substack{V \, \subset \, \calu_j \\ V^TV = I_k}} \rho(V)
= \max_{\substack{V \, = \, U_jZ \, \in \, \R^{p \times k} \\ V^TV = I_k}} \frac{\tr(V^T\!AV)}{\tr(V^TBV)} = \max_{Z^TZ = I_k} \frac{\tr(Z^TH_jZ)}{\tr(Z^TK_jZ)},
\end{equation}
which is a TR problem for the pair of projected matrices $(H_j,\, K_j) = (U_j^TAU_j, \, U_j^TBU_j)$. 

Given the characterization of the global maximizer of TR (see Proposition~\ref{prop:char}), one significant quantity to measure the accuracy of $V$ as approximate solution is the $p\times k$ residual matrix
\begin{equation}\label{eq:res-mat}
R = AV - \rho_j \, BV - V\Lambda = (I - VV^T)(A - \rho_jB)\,V,
\end{equation}
where $\Lambda = V^T\!AV - \rho_j \, V^T\!BV$ is diagonal, with $\tr(\Lambda) = 0$, and contains the $k$ largest eigenvalues of $H_j - \rho_j K_j$. In addition, the matrix $\Lambda$ is a Rayleigh quotient for $A - \rho_j B$ and, therefore, satisfies the minimum residual condition
\[
\Lambda = \argmin_{E \, \in \, \mathbb R^{k\times k}} \, \|AV-\rho_j \, BV-VE\|_F.
\]
It is easy to see that $R^TV = 0$, i.e., $R$ satisfies a Galerkin condition (cf., e.g., \cite[Sec.~4.3]{saad2011numerical}). As stated in Proposition~\ref{prop:char}, $R = 0$ is a necessary optimality condition for $(V,\rho_j)$ to be the global solution of \eqref{eq:TR.MaxProblem}. 
Note that $R$ is also parallel to the gradient of the trace ratio in the Grassman manifold $\{VQ\, :\, V^TV = I_k, \ Q \in \mathbb{R}^{k\times k},\, Q^T Q = I_k\}$ and for the Euclidean metric (see, e.g., \cite[Sec.~2.5.3]{edelman1998geometry}).

\subsection{Subspace expansion} \label{sec:expa}
Subspace expansion aims to build a search subspace rich enough to contain the solution to the original problem. Due to memory constraints, it is reasonable to let the subspace size vary from $\kmin \ge k$ to $\kmax \le p$, with $\kmin < \kmax$. In Davidson type methods, the current basis is typically expanded by including some information from preconditioned residual vectors. The next proposition shows that the trace ratio value remains non-decreasing, regardless of the augmentation strategy.

\begin{proposition}\label{prop:mon-rho}
Consider two subspaces $\calu_j\subset\calu_{\wt j}$, where the dimensions satisfy $k \le \kmin \le j < \wt j \le p$. Then 
\[
\max_{V\,\subset\,\calu_{j},\,V^TV = I_k}\rho(V) \le \max_{V\,\subset\,\calu_{\wt j},\,V^TV = I_k}\rho(V).
\]
\end{proposition}
\begin{proof}
The thesis immediately follows from the inclusion of the two feasible regions. 
\end{proof}

This result is also valid for $\wt j = p$, meaning that, as $j$ increases, the search space will eventually include the solution of the original TR \eqref{eq:TR.MaxProblem}.
Since reaching $p$ is unwanted, it might happen that the method does not converge to the optimal $\Vopt$ with a very poor initial space and expansion technique. However, this is a rare situation, which is not observed in our experiments.


Davidson's methods for eigenvalue problems typically exploit two main expansion techniques: one involves selecting a single residual vector (see, e.g., \cite{morgan1986generalizations,saad2011numerical,stathopoulos1998dynamic}), while the other is to keep the whole residual matrix, for a block expansion (see, e.g., \cite{morgan1990davidson,crouzeix1994davidson}).
We also exploit a block method, which extends the current search subspace $U_j$ by the $\bsize$ largest singular values of the residual matrix $R$, stored in $U_\bsize = [\bu_{j+1},\dots,\bu_{j+\bsize}]$. Here $1\le\bsize\le k$, where $\bsize$ is usually a common divisor of $\kmin$ and $\kmax$, which includes $\bsize = 1$. Without loss of generality, we assume $j + \bsize \le \kmax$; if this is not the case for some iterations, we take $\bsize = \kmax - j$. Note that, in our implementation, the block size $\bsize$ is further reduced by considering only those singular vectors $\bu_i$, $i=j+1,\dots,j+\bsize$, that correspond to singular values $\sigma_i \ge 10^{-4}\cdot\sigma_1$.

Computationally, this block expansion amounts to a singular value decomposition (SVD) of a (thin) $p \times k$ matrix, with complexity $\calo(pk^2)$. The rationale behind the choice of the largest singular vectors is to condense the information contained in the residual matrix, which will diminish as $R$ approaches $0$. In addition, the SVD of $R$ will systematically provide the spectral norm of the residual matrix, which can be used as a stopping criterion. 

It holds that $U_\bsize \perp U_j$ (i.e., the column vectors of $U_\bsize$ are orthogonal to those of $U_j$), since $R$ satisfies $R\perp \calu_j$. 
However, it is recommended to perform a reorthogonalization step (of the form $U_\bsize \leftarrow U_\bsize - U_j(U_j^TU_\bsize)$), followed by the QR decomposition of $U_\bsize$, to keep the basis numerically orthogonal. When $U_j$ is $p\times j$, this takes $\calo(pj\bsize)$ work. When $\bsize = 1$, the QR decomposition reduces to the renormalization of $U_1 = \bu_{j+1}$.


Finally, our method can be implemented without preconditioning the residual matrix. 
In standard eigenvalue problems preconditioning is frequently used. Otherwise, Davidson's method without the preconditioning of residuals is equivalent to (but more expensive than) the Arnoldi method  (see, e.g., \cite{stathopoulos1998dynamic}), where the extra cost is in the computation of an eigenvalue decomposition any time a new vector is added to the search subspace. 

\subsection{Restart}\label{sec:restart}
As $j$ increases, solving the projected trace ratio problem becomes increasingly expensive.
In addition, the $\calo(pj)$ storage requirement also increases.
Generally, our intention is $\kmax \ll p$.
Therefore, the method needs to be restarted, that is, the dimension of the search space $\calu_j$ is reduced to minimum dimension $\kmin$, once $j$ has reached the maximum dimension $\kmax$.

A typical restarting method for eigenvalue problems is to obtain $\kmin \ge k$ eigenvectors from the $\kmax\times\kmax$ projected problem, containing the $k$ eigenvectors of interest. This strategy is adopted in the Davidson type method for the GEP (cf.~Section~\ref{sec:trsub-gep}). The equivalent choice for the trace ratio problem would be to consider the projected TR \eqref{eq:proj-tr} for $Z\in \mathbb R^{\kmax \times \kmin}$. However, for $k \le \kmin \le p$, the solution of the $k$-dimensional TR is generally not included in the $\kmin$-dimensional solution (cf., e.g., \cite{ferrandi2022trace}). Additionally, as shown in \cite[Prop.~7]{ferrandi2022trace}, the optimal trace ratio value is non-increasing as the size of its solution increases. Therefore the new trace ratio value would usually be smaller than the one corresponding to $Z\in \mathbb R^{\kmax \times k}$, causing a non-monotonic behavior of $\rho$ throughout the iterations.

Nevertheless, it is possible to obtain a restarting strategy that keeps the trace ratio non-decreasing. Let $\calu_\kmax$ be a search subspace of dimension $\kmax$ and $\calu_\kmin$ be the subspace after restarting. In a typical situation, $\calu_\kmin\subset\calu_\kmax$, and therefore, from Proposition~\ref{prop:mon-rho},
\begin{equation}\label{eq:inclusion}
\max_{V\,\subset\, \calu_\kmin, \,V^TV = I_k}\rho(V) \le \max_{V\,\subset\, \calu_\kmax, \,V^TV = I_k}\rho(V) = \rho_\kmax.
\end{equation}
Since we aim at keeping the monotonicity of the trace ratio throughout the iterations, a subspace $\calu_\kmin$ which attains equality in \eqref{eq:inclusion} would be the ideal candidate for restarting. 
A choice that satisfies this requirement is the following. Let $U_\kmax$ be an orthonormal basis for $\calu_\kmax$, and consider the $k$-dimensional projected TR \eqref{eq:proj-tr}, for the pencil $(H_\kmax,\,K_\kmax) = (U_\kmax^TAU_\kmax, \,U_\kmax^TBU_\kmax)$, where the matrices of size $\kmax\times \kmax$ and the solution is $(Z, \rho_\kmax)$. Now expand $Z$ with the next $\kmin - k$ eigenvectors of $H_\kmax - \rho_\kmax K_\kmax$, denoted by $Z_{\kmin-k}$.
Let $Z_{\kmin} = [Z\ \, Z_{\kmin-k}]$. The search subspace for the restart is chosen as 
\[U_{\kmin} = U_\kmax Z_\kmin, \quad\text{with}\quad \calu_\kmin\subset\calu_\kmax.\]
Simple substitutions show that
\begin{align*}
\max_{V\,\subset\, \calu_\kmin,\, V^TV = I_k}\rho(V) = \max_{V = U_\kmax Z_\kmin C,\ C^TC = I_k}\frac{\tr(C^TZ_\kmin^TH_\kmax Z_\kmin C)}{\tr(C^TZ_\kmin^TK_\kmax Z_\kmin C)} \le \rho_\kmax, 
\end{align*}
where the upper bound is given by \eqref{eq:inclusion}. From the definition of $Z$, it is readily seen that $C = [I_k\ 0]^T \in \R^{\kmax \times k}$ achieves $\rho_\kmax$, therefore attaining equality in the relation \eqref{eq:inclusion}. Additionally, for this value of $C$, the new tentative solution is the same as the one coming from $\calu_\kmax$, i.e., $V = U_\kmax Z_{\kmin}C = U_\kmax Z$. This also means that, immediately after this type of restarting, it is not necessary to run the extraction phase to get a new vector for the expansion. 

The matrix $C$ is not uniquely determined, so the equality between the tentative solutions is always up to $k\times k$ orthonormal transformations. Other choices can be made to replace $Z_{\kmin - k}$. The above argument shows that, to attain equality in \eqref{eq:inclusion}, it is sufficient for the new basis to be in the span of $U_\kmax$ and contain $U_\kmax Z$. 

\subsection{Algorithm}
Adding all ingredients, a pseudocode for a subspace method is given in Algorithm~\ref{algo:traceratio-subspace}. Each outer iteration requires two matrix-vector products, one with $A$ and one with $B$, for the extraction phase. 
The multiplication with $A$ and $B$ may have cost $\calo(p^2)$, but it may also be more efficient (cf.~Section~\ref{sec:classification}).
As previously remarked, the cost of computing the left singular vectors of $R$ is $\calo(pk^2)$. The reorthogonalization cost in Line~11 is $\calo(p\kmax\bsize)$.
Therefore, for large matrices, the main cost for an outer iteration is due to the matrix-vector products.

\begin{algorithm}
\caption{Subspace method for large-scale trace ratio problems}
\label{algo:traceratio-subspace}
{\bf Input:} Symmetric $A \in \R^{p \times p}$, SPD $B \in \R^{p \times p}$, $k < p$,
maximum and minimum dimensions with $k \le \kmin < \kmax \le p$, block size $1\le\bsize\le k$, initial $p \times \kmin$ matrix $U_\kmin$ with orthonormal columns, tolerance {\sf tol}.\\
{\bf Output:} $p \times k$ matrix $V$ with orthonormal columns maximizing the trace ratio. \\
\begin{tabular}{ll}
{\footnotesize 1:} & Compute $AU_\kmin$, $BU_\kmin$ and $H_\kmin = U_\kmin^T\!AU_\kmin$, $K_\kmin = U_\kmin^T\!BU_\kmin$; $j = \kmin$\\
{\footnotesize 2:}& {\bf for} $i = 1, 2, \dots$ \\
{\footnotesize 3:} & \phantom{M} Run Algorithm~\ref{algo:trsub-traceratio} on $(H_j, K_j)$, with output $(Z,\rho_j)$\\ 
{\footnotesize 4:} & \phantom{M} $V = U_jZ$, \ $R=AV-\rho_j \, BV -V (Z^T(H_j - \rho_j \, K_j)Z)$ (cf.~\eqref{eq:res-mat})\\
{\footnotesize 5:} & \phantom{M} Compute $U_\bsize = [\bu_{j+1},\dots,\bu_{j+\bsize}]$
largest left singular vectors of $R$\\
{\footnotesize 6:} & \phantom{M} {\bf if} $\|R\| < {\sf tol}$,  \ {\bf return}, \ {\bf end} \\
{\footnotesize 7:} & \phantom{M} {\bf if}  $j = \kmax$:\\
{\footnotesize 8:} & \phantom{MM} Store $\kmin$ leading eigenvectors of $H_\kmax-\rho_\kmax K_\kmax$ in $Z_\kmin = [Z \ \ Z_{\kmin-k}]$\\
{\footnotesize 9:} & \phantom{MM} Set $j = \kmin$, shrink $U_j = U_\kmax Z_\kmin$,  $AU_j = AU_\kmax Z_\kmin$, $BU_j = BU_\kmax Z_\kmin$\\
 & \phantom{MMm} $H_j = Z_\kmin^T\!H_\kmax Z_\kmin$, $K_j = Z_\kmin^T\!K_\kmax Z_\kmin$ \\
{\footnotesize 10:} & \phantom{M} {\bf end}\\
{\footnotesize 11:} & \phantom{M} Reorthogonalize $U_\bsize$ against $U_j$, expand $U_{j+\bsize} = [U_j\ \ U_\bsize]$\\
{\footnotesize 12:} & \phantom{M} Compute $A\,U_\bsize$, $B\,U_\bsize$ and $U_{j+\bsize}^TA\,U_\bsize$, $U_{j+\bsize}^TB\,U_\bsize$ \\
{\footnotesize 13:} & \phantom{M} $j = j + \bsize$\\
{\footnotesize 14:} & {\bf end}
\end{tabular} 
\end{algorithm}

Given that $R = 0$ is a necessary optimality condition, a reasonable stopping criterion may be based on the spectral norm of the residual matrix, as in Line~6 of Algorithm~\ref{algo:traceratio-subspace}. This is a convenient choice since the expansion phase requires the SVD of $R$. (The expansion is anticipated to allow the computation of the spectral norm. If another norm is selected as a stopping criterion, the expansion may be postponed.)

The extraction phase (Line~3 of Algorithm~\ref{algo:traceratio-subspace}) consists in running Algorithm~\ref{algo:trsub-traceratio} with $j\times j$ matrices. In our experiments, the trace ratio value makes little progress in the last iterations of the method, while the residual norm keeps decreasing. Therefore, to reduce the number of inner iterations, for each projected TR we suggest starting Algorithm~\ref{algo:trsub-traceratio} from the latest available value of $\rho_{j-1}$, and computing the initial matrix $Z$ as the solution to the eigenvalue problem for $H_j-\rho_{j-1} K_j$. The inner stopping criterion of Algorithm~\ref{algo:trsub-traceratio} is given by $\|(I_k - ZZ^T)(H_j - \rho_j K_j)Z\| < 10^{-8}$, where $\rho_j = \rho(Z)$ is the trace ratio in the current iterate $Z$. 
This expression corresponds to the spectral norm of a residual matrix; cf.~\eqref{eq:res-mat}. 

In the computation of $R$ (cf.~Line~4), the trace ratio $\rho_j$ is an approximate value for the maximum of the $j\times j$ projected TR \eqref{eq:proj-tr}. In particular, $Z$ is only an approximation to the eigenvectors corresponding to the $k$ largest eigenvalues of $H_j - \rho_j K_j$, and, as a consequence, $\Lambda = Z^T(H_j - \rho_j K_j)Z = V^T(A-\rho_j B)V$ is generally not diagonal. This does not contradict the definition of the residual matrix in Section~\ref{sec:trsub-extra}, since there $\rho_j$ is the (exact) maximum of the corresponding projected TR.    

Algorithm~\ref{algo:traceratio-subspace} is initialized with a random matrix $U_\kmin$ with orthonormal columns. When a crude estimate of the optimal trace ratio value $\wh \rho \approx \rhopt$ is available, one may use a basis for the $\kmin$-dimensional Krylov subspace $\calk_{\kmin}(A-\wh \rho B, \bv_1)$ as starting point (where $\bv_1\in \mathbb R^p$ is randomly chosen on the unit sphere). However, this estimate may be difficult to obtain; in fact, our subspace method is exactly meant to compute this quantity and associated $\Vopt$.

\section{Analysis}\label{sec:trsub-analysis}
The output quality of Algorithm~\ref{algo:traceratio-subspace} is now analyzed from a theoretical viewpoint. Throughout the section, we assume that there is a gap between the $k$th and the $(k+1)$st eigenvalue of $A-\rhopt B$, so that the solution to TR $\Vopt$ is uniquely determined. 
Let $\rhopt$ be its corresponding maximum, and let $U$ be an orthonormal basis for the current search subspace $\calu$, of dimension $\kmin\le j \le\kmax$. Let $(Z, \rho)$ be the solution to the projected TR \eqref{eq:proj-tr} when the search subspace is $\calu$, and let $V = UZ$ be the approximate solution to the original TR problem. As in Section~\ref{sec:trsub-extra}, define $H = U^T\!AU$ and $K = U^TBU$. 

Define the subspaces $\calv = {\rm span}(V)$ and $\calv^\ast = {\rm span}(\Vopt)$. The goal is to determine a bound for the eigenvalues of $H - \rho K$ and the angle $\angle(\calv,\calv^\ast)$ in terms of either the spectral norm of the residual matrix or the sine of the angle $\angle(\calu, \calv^\ast)$. The latter is defined as (see, e.g., \cite{jia2000analysis}):
\begin{equation}\label{eq:sine}
\sin\theta := \sin\angle(\calu, \calv^\ast) = \sin(U,\Vopt) = \|U_\perp^T \Vopt\|,
\end{equation}
where $U_\perp$ is an orthonormal basis for the orthogonal complement of $\calu$. Since $U$ and $\Vopt$ have different numbers of columns, the sine is not symmetric in its arguments (cf., e.g., \cite{knyazev2002principal}). Also note that \eqref{eq:sine} is well defined for any orthonormal bases of $\calu$ and $\calv^\ast$. Convergence results may be obtained under the hypothesis that $\sin\theta\to 0$, which means that the search subspace will eventually become rich enough to contain the solution of TR.  

\subsection{A bound for the approximate eigenvalues}
In the Rayleigh--Ritz extraction for standard eigenvalue problems, it is known that some Ritz values converge to the desired eigenvalues, as $\theta\to 0$ (see, e.g., \cite[Thm.~4.1]{jia2000analysis} and \cite[Sec.~4.4]{Ste01} for one approximate eigenvalue). 
The TR problem has the additional challenges that the optimal trace ratio value $\rhopt$ is not known in advance, and the underlying eigenvalue problem changes throughout the iterations. Nevertheless, it is still possible to extend this result, by showing that there exist $k$ eigenvalues of $H - \rho K$ converging to the $k$ largest eigenvalues of $A - \rhopt B$, as $\theta\to 0$. 

To this aim, the next proposition shows a bound similar to the one derived in \cite[Thm.~4.1]{jia2000analysis}. 

\begin{proposition}\label{prop:trsub-eig}
Let $(\Vopt,\rhopt)$ be the solution to the TR \eqref{eq:TR.MaxProblem}, with $(A-\rhopt B)\Vopt = \Vopt\Lambda^\ast$, and let $\rho$ be the solution to the projected TR \eqref{eq:proj-tr}, restricted to the subspace $\calu = {\rm span}(U)$, where $U$ has orthonormal columns. Let $\Vopt$ be decomposed as
\begin{equation}\label{eq:decoV1}
\Vopt = UU^T \Vopt + U_\perp U_\perp^T \Vopt =: UE + U_\perp F.
\end{equation}
Here, $E$ is $j \times k$ and $U_\perp$ is $p \times (p-j)$.
In addition, let $H-\rho K$ be the projection of $A-\rho B$ onto $U$, i.e., $H-\rho K = U^T(A-\rho B)U$. Let $\theta$ be as in \eqref{eq:sine}, and define the $j\times k$ matrix $\wh E = E\,(E^TE)^{-1/2}$ with orthonormal columns. Then there exists a matrix 
\begin{equation}\label{eq:pertG}
G = [(\rho -\rhopt)\,K\wh E + U^T(A-\rhopt B)\,U_\perp F\,(E^T\!E)^{-1/2}]\,\wh E^T
\end{equation}
with
\[
\|G\|\le (\rhopt -\rho) \, \|B\| + \frac{\sin\theta}{\sqrt{1-\sin^2\theta}} \, \|A-\rhopt B\|,
\]
and such that $((E^TE)^{1/2}\,\Lambda^\ast\,(E^TE)^{-1/2}, \,\wh E)$ is an eigenblock of $H-\rho K + G$, with the same eigenvalues as $\Lambda^\ast$.
\end{proposition}
\begin{proof}
From the optimality of $(\Vopt,\rhopt)$ it follows that $U^T(A-\rhopt B)(UE + U_\perp F) = E\Lambda^\ast$, so
\begin{align*}
(H-\rhopt K)E + U^T(A-\rhopt B)\,U_\perp F = E\Lambda^\ast.
\end{align*}
A perturbation $G$ should satisfy
\[
(H-\rhopt K + G)\,\wh E = \wh E\,(E^TE)^{\frac12}\,\Lambda^\ast \,(E^TE)^{-\frac12}.
\]
One possible solution is given by \eqref{eq:pertG}. An upper bound for the spectral norm of $G$ is the following. From the triangular inequality, the unitarily invariance of the spectral norm, and submultiplicativity, we obtain
\[
\|G\|\le (\rhopt -\rho) \, \|B\| + \|A-\rhopt B\|\cdot\|U_\perp^T\Vopt\|\cdot\|(E^T\!E)^{-1}E^T\|.
\]
Let us consider the SVD of $E = U_E\Sigma_EV_E^T$. Then $(E^TE)^{-1}E^T = V_E\Sigma_E^{-1}U_E^T.$ Its spectral norm corresponds to the reciprocal of the smallest singular value of $U^T\Vopt$, which is the cosine of the largest angle between $U$ and $\Vopt$. Then  
\[
\|U_\perp^T\Vopt\|\cdot\|(E^T\!E)^{-1}E^T\| = \frac{\sin\theta}{\sqrt{1-\sin^2\theta}},
\]
which completes the proof.
\end{proof}

This proposition holds for any $\rho\le \rhopt$. In particular, when $\rho = \rhopt$, the result reduces to \cite[Thm.~4.1]{jia2000analysis}. Nevertheless, the solution to the projected TR \eqref{eq:proj-tr} is the closest approximation to $\rhopt$ given $U$, making the bound for $\|G\|$ tighter. More importantly, $\rho$ plays a role in showing that the term $\rhopt - \rho$ vanishes as the search subspace approaches the solution to the trace ratio. 

\begin{proposition}\label{prop:deltarho}
Let $(\Vopt,\rhopt)$ be the solution to the TR \eqref{eq:TR.MaxProblem}, and  $\theta$ be as in \eqref{eq:sine}. Let $\rho$ be the solution to the projected TR \eqref{eq:proj-tr}, where the problem is projected onto the subspace $\calu = {\rm span}(U)$ and $U$ has orthonormal columns. Then
\[
\rho \to \rhopt \quad\text{as} \quad \theta\to 0.
\]
\end{proposition}
\begin{proof}
Let $\Vopt$ be written in the form \eqref{eq:decoV1}. Then it follows that
\begin{align}
\tr((\Vopt)^TB\Vopt)\,(\rhopt - \rho) &= \tr((\Vopt)^T(A-\rho B)\Vopt)\nonumber\\
&= \tr((UE + U_\perp F)^T(A-\rho B)(UE + U_\perp F)) \nonumber\\
&= \tr(E^T(H-\rho K)E) + \tr(F^TU_\perp^T(A-\rho B)(UE + \Vopt)).\label{eq:delta-rho-deco}
\end{align}
Since $F = U_\perp^T\Vopt$, the second term of the right-hand side vanishes when $\theta \to 0$. 

Given that $\Vopt$ has orthonormal columns, as $\theta\to 0$, the columns of $E$ approach orthogonality. Therefore, $E$ can be split into $E = \wt E + \Delta E$, where $\wt E$ has orthonormal columns. As one possible decomposition, given the SVD of $E = U_E\Sigma_E V_E^T$, we may define $\wt E = U_EV_E^T$ and $\Delta E = U_E\,(\Sigma_E - I_k)\,V_E^T$, with $\|\Delta E\|\to 0$. In addition, $\calo(\|\Delta E\|) = \calo(\|F\|^2)$, which easily follows from the relation $E^TE + F^TF = I_k$ (which is equivalent to $-F^TF = \wt E^T\Delta E + \Delta E^T \wt E + \Delta E^T\Delta E$). 

The conclusion is that 
\begin{equation}\label{eq:delta-rho-deco2}
\tr(E^T(H-\rho K)E) = \tr(\wt E^T(H-\rho K)\wt E) + \calo(\|F\|^2),
\end{equation}
where $\tr(\wt E^T(H-\rho K)\wt E) \le 0$ from the characterization of the solution to TR (cf.~Proposition~\ref{prop:char}). Since $\tr((\Vopt)^TB\Vopt)$ is bounded from below by the $k$ smallest eigenvalues of $B$, and $\rhopt - \rho \ge 0$, from \eqref{eq:delta-rho-deco} and \eqref{eq:delta-rho-deco2} it follows that $\rhopt - \rho \to 0$ as $\theta\to 0$, which completes the proof.
\end{proof}

Proposition~\ref{prop:trsub-eig} shows that the $k$ largest eigenvalues of $A - \rhopt B$ are equal to $k$ eigenvalues of $H - \rho K + G$, where $G$ is given by \eqref{eq:pertG}. Then the following result holds.
\begin{corollary}
There exists a subset of $k$ eigenvalues of $H - \rho K$ converging to the $k$ largest eigenvalues of $A - \rhopt B$, as $\theta\to 0$.
\end{corollary}
\begin{proof}
From Bauer--Fike Theorem (see, e.g., \cite[p.~18]{Par74}) it follows that for any $\lambda^\ast_i$ eigenvalue of $\Lambda^\ast$, there is an eigenvalue $\lambda_{\pi_i}$ of $H - \rho K$ such that
\[
\vert\lambda_{\pi_i} - \lambda^\ast_i\vert \le \|G\|.
\]
In addition, from Proposition~\ref{prop:trsub-eig} and Proposition~\ref{prop:deltarho}, it holds that $\|G\|\to 0$ as $\theta\to 0$. This completes the proof. 
\end{proof}

This result does not guarantee that the converging eigenvalues are actually the $k$ largest eigenvalues of $H-\rho K$, which is in line with the conclusion of Jia and Stewart \cite{jia2000analysis} about Ritz values converging to the desired eigenvalues.





\subsection{A bound for the approximate solution} We now provide a bound on the angle between the approximate solution and the solution to TR, i.e., $\angle(\calv,\calv^\ast)$, in terms of $\angle(\calu,\calv^\ast)$ and $\rhopt - \rho$. This result extends the bound in \cite[Thm.~2]{stewart2001generalization} for approximate eigenspaces. It is based on the ${\rm sep}$ function, which measures the separation between the spectra of two matrices (see, e.g., \cite[Eq.~(2.3)]{jia2000analysis}): 
\begin{equation}\label{eq:sep}
{\rm sep}(A,B) := \min_{\|P\| = 1}\|PA - BP\|,
\end{equation}
where $A$ is $k\times k$, $B$ is $(p-k)\times(p-k)$, and $P$ is $(p-k)\times k$, for any $1\le k< p$. It can be shown (see, e.g., \cite{jia2000analysis}) that ${\rm sep}$ is bounded from above by $\min_{i,j}\vert\lambda_i(A) - \lambda_j(B)\vert$, but can be much smaller (see, e.g., \cite[Ch.~V, Ex.~2.4]{SSu90}). 

\begin{proposition}
 Let $(\Vopt,\rhopt)$ be the solution to the TR \eqref{eq:TR.MaxProblem}, with $(A-\rhopt B)\Vopt = \Vopt\Lambda^\ast$, and let $(V,\rho)$ be the solution to the projected TR \eqref{eq:proj-tr}, restricted to the subspace $\calu = {\rm span}(U)$, where $U$ has orthonormal columns. In addition, let $H-\rho K$ be the projection of $A-\rho B$ onto $U$, i.e., $H-\rho K = U^T(A-\rho B)U$, and $\theta$ be as in \eqref{eq:sine}. 

 Finally, consider an orthonormal basis for $\mathbb R^p$, split into three parts $[V\ Q\ U_\perp]$, where the search subspace is decomposed as $U = [V\ Q]$; $Q$ spans the orthogonal complement of $V$ in $\calu$; $U_\perp$ spans the orthogonal complement of $U$ in $\mathbb R^p$. 
 
Then
\begin{equation}\label{eq:sine-VVstar}
\sin^2(V, \Vopt) \le \sin^2\theta + \left(\frac{
\eta \, \sin\theta + (\rhopt - \rho) \, \|B\| \, (1 + \sin\theta)
}{{\rm sep}(\Lambda^\ast,M)}\right)^2,
\end{equation}
where $\eta = \|U^T(A-\rhopt B)\,U_\perp\|$, $M = Q^T(A-\rho B)\,Q$, and ${\rm sep}$ is defined in \eqref{eq:sep}.
\end{proposition}
\begin{proof}
Our starting point is the relation $(A-\rhopt B)\,\Vopt = \Vopt\Lambda^\ast$, written in the basis $[V\ Q\ U_\perp]$:
\begin{align*}
\begin{bmatrix}
V^T \\ Q^T \\ U_\perp^T
\end{bmatrix}
(A -\rhopt B)
[V\ Q\ U_\perp]
\begin{bmatrix}
V^T \\ Q^T \\ U_\perp^T
\end{bmatrix}\Vopt = 
\begin{bmatrix}
B_{11} & B_{12} & B_{13}\\
B_{12}^T & B_{22} & B_{23}\\
B_{13}^T & B_{23}^T & B_{33}\\
\end{bmatrix}
\begin{bmatrix}
E_1 \\ E_2 \\ F
\end{bmatrix} = \begin{bmatrix}
E_1 \\ E_2 \\ F
\end{bmatrix}\Lambda^\ast.
\end{align*}
The quantity $\rho\ [V^T\ \ Q^T]^TB\,[V \ \ Q]$ is added and subtracted to exploit the fact that, by the definition of $V$ and $Q$, $[V^T\ \ Q^T]^TB\,[V \ \ Q] = {\rm diag(\Lambda, M)}$ is block diagonal, with $\Lambda$ as in the definition of the residual matrix \eqref{eq:res-mat} and $M = Q^T(A-\rho B)\,Q$. Then the first $2\times 2$ block becomes
\begin{equation}\label{eq:2by2}
\begin{bmatrix}
B_{11} & B_{12}\\
B_{12}^T & B_{22}\\
\end{bmatrix}
\begin{bmatrix}
E_1 \\ E_2
\end{bmatrix}
-
\begin{bmatrix}
E_1 \\ E_2
\end{bmatrix}\Lambda^\ast
=
\begin{bmatrix}
\Lambda E_1 - E_1\Lambda^\ast \\
M E_2 - E_2\Lambda^\ast\\
\end{bmatrix} 
+ (\rho - \rhopt)
\begin{bmatrix}
V^T \\ Q^T
\end{bmatrix}
B \,
(VE_1 + QE_2).
\end{equation}
Upper and lower bounds for \eqref{eq:2by2} can be determined using ${\rm sep}$ \eqref{eq:sep}, as follows.
\begin{align*}
\|E_2\|\ {\rm sep}(\Lambda^\ast, M) &\le \|M E_2- E_2\Lambda^\ast + (\rho-\rhopt)\,Q^TB\,(VE_1 + Q E_2)\|\\
&\qquad+ (\rhopt - \rho)\,\|Q^TB\,(VE_1 + Q E_2)\|\\
&\le \Big\|
\begin{bmatrix}
B_{11} & B_{12}\\
B_{12}^T & B_{22}\\
\end{bmatrix}
\begin{bmatrix}
E_1 \\ E_2
\end{bmatrix}
-
\begin{bmatrix}
E_1 \\ E_2
\end{bmatrix}\Lambda^\ast
\Big\| + (\rhopt - \rho)\,\|B\|\,(1 + \|F\|)\\
&\le \Big\|
\begin{bmatrix}
B_{13} \\ B_{23}
\end{bmatrix}  
\Big\| \cdot \|F\| + (\rhopt - \rho)\,\|B\|\,(1 + \|F\|).
\end{align*}
The second inequality comes from the fact $VE_1 + QE_2 = \Vopt - U_\perp F$. Let us define $\eta := \|[B_{13}^T\ B_{23}^T]^T\| = \|U^T(A-\rhopt B)\,U_\perp\|$. For any $E, F$ with equal number of columns, it holds that $\|[E^T\ F^T]^T\| \le \|E\|^2 + \|F\|^2$. Then
\begin{align*}
\Big\|\begin{bmatrix}
E_2 \\ F 
\end{bmatrix}\Big\|^2 \le \|F\|^2 + \left(\frac{
\eta \, \|F\| + (\rhopt - \rho)\,\|B\|\,(1 + \|F\|)
}{{\rm sep}(\Lambda^\ast, M)}\right)^2.
\end{align*}
It is not difficult to see that $\|[E_2^T\ F^T]^T\| = \|V_\perp^T\Vopt\| = \sin(V,\Vopt)$ and $\|F\| = \|U_\perp^T\Vopt\| = \sin(U, \Vopt)$, which concludes the proof.
\end{proof}

As for Proposition~\ref{prop:trsub-eig}, when $\rho = \rho^\ast$, the bound in \eqref{eq:sine-VVstar} is identical to the one derived in \cite{stewart2001generalization} for Ritz vectors in a standard eigenvalue problem. 
Since ${\rm sep}(M, \Lambda^\ast)$ might also approach zero, as $\theta\to 0$, this proposition does not immediately guarantee the convergence of $V$ to the solution $\Vopt$. For this reason, we need the notion of {\em uniform separation condition} (see, e.g., \cite{jia2000analysis}), i.e., we assume that there exists a $\delta^\ast > 0$, independent of $\theta$, such that ${\rm sep}(M, \Lambda^\ast) \ge \delta^\ast > 0$ for all $M$. This additional hypothesis, combined with the convergence of $\rho$ (cf.~Proposition~\ref{prop:deltarho}), is sufficient to conclude that $\sin(V,\Vopt)\to 0$ as $\theta\to 0$.

\subsection{A bound for the approximate solution in terms of the residual matrix}
Given the current approximation to the solution to TR $(V,\rho)$, the spectral norm of the corresponding residual matrix \eqref{eq:res-mat} is used as the stopping criterion in Algorithm~\ref{algo:traceratio-subspace}. In the same spirit of \cite[Thm.~6.1]{jia2000analysis}, we show that the norm of the residual matrix occurs in another upper bound for $\sin(V,\Vopt)$.
\begin{proposition}\label{prop:angle-bound}
Let $( V, \rho)$ be an approximate solution to the TR \eqref{eq:TR.MaxProblem}, with residual $R = (A - \rho B) V -  V \Lambda$, as in \eqref{eq:res-mat}. Let $(\Vopt, \rhopt)$ be the exact solution to the TR. Then 
\begin{equation}\label{eq:angle-bound}
\sin ( V, \Vopt) 
\le\frac{\|R\| + (\rhopt - \rho)\,\|(\Vopt_\perp)^T B V\|}{{\rm sep}( \Lambda, M^\ast)} \le\frac{\|R\| + (\rhopt - \rho)\,\|B\|}{{\rm sep}( \Lambda, M^\ast)},
\end{equation}
where $\Vopt_\perp$ is the orthogonal complement of $\Vopt$; $M^\ast$ is the eigenblock of $A - \rhopt B$ corresponding to $\Vopt_\perp$, and ${\rm sep}$ is defined as in \eqref{eq:sep}.
\end{proposition}
\begin{proof}
By the definitions of $R$ and $\Vopt_\perp$, the projection of $R$ onto $\Vopt_\perp$ can be rewritten as
\[
(\Vopt_\perp)^TR = M^\ast(\Vopt_\perp)^T  V - (\Vopt_\perp)^T  V\Lambda + (\rhopt - \rho)\,(\Vopt_\perp)^TB V.
\]
Therefore
\[
{\rm sep}(\Lambda, M^\ast)\,\|(\Vopt_\perp)^T  V\| \le \|(\Vopt_\perp)^TR\| + (\rhopt - \rho)\,\|(\Vopt_\perp)^T B V\| \le \|R\| + (\rhopt - \rho)\,\|B\|. 
\]
\end{proof}

This proposition suggests that reasonable stopping criteria for TR could be based on either the norm of the residual matrix, or the difference between consecutive values of $\rho$. Nevertheless, if there is a small gap between the $k$th and the $(k+1)$st eigenvalue of $A-\rhopt B$, the denominator in \eqref{eq:angle-bound} may become quite small, making the bound useless.

The behavior of $\rhopt - \rho$ and $\|R\|$ is illustrated in the following example. 

\begin{example}\label{ex:angle-bound}
\rm Let $B \in \mathbb R^{p\times p}$ be a random SPD matrix, and let $A = H_AH_A^T \in \mathbb R^{p\times p}$, where $H_A$ is $p\times (k+1)$ with uniformly distributed entries in $(0,1)$; set $p = 100$ and $k = 5$. In addition, $H_A$ is centered so that its columns have zero mean. Algorithm~\ref{algo:traceratio-subspace} is run with $\kmin = 10$ and $\kmax = 25$. The iterations stop when $\|R\| \le 10^{-6}$. The solution $(\Vopt, \rhopt)$ is provided by the output of Algorithm~\ref{algo:trsub-traceratio}, with tolerance $10^{-8}$. 
\begin{figure}[htb!]
\centering
\includegraphics[width=0.7\textwidth]{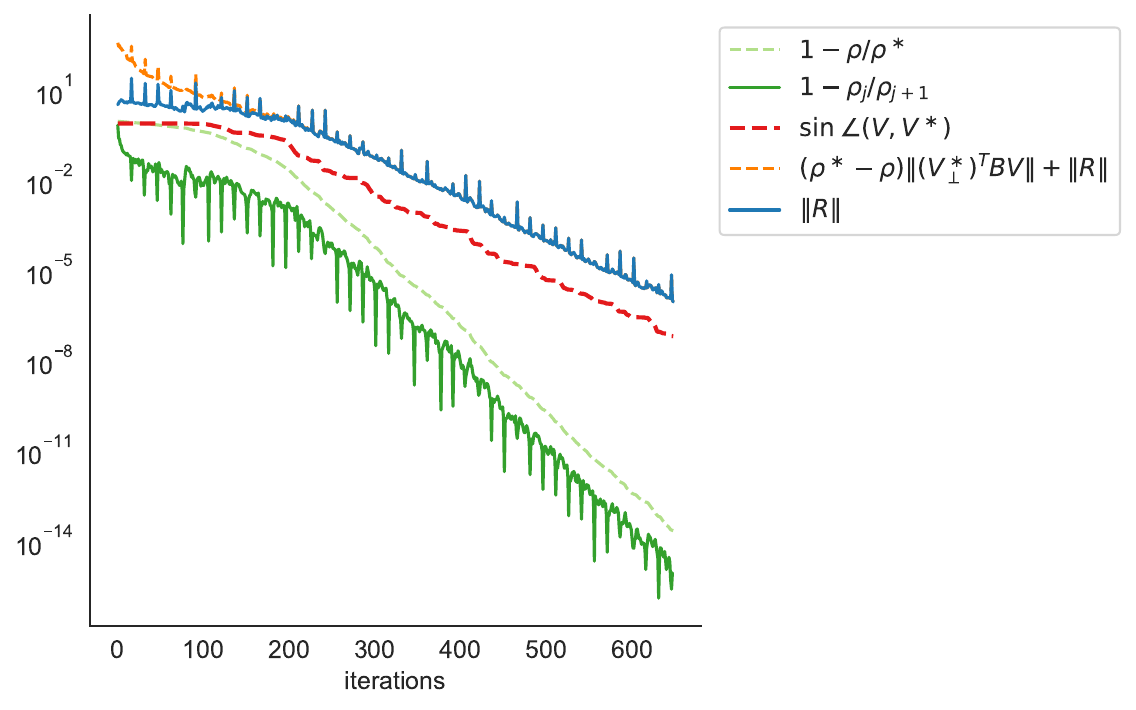}
\caption{Numerical and theoretical quantities to check the convergence of the approximate solution to TR to the exact solution.}
\label{fig:res-analysis}
\end{figure}
Figure~\ref{fig:res-analysis} shows the values of $\sin ( V, \Vopt)$ in the iterations, along with the numerator of the tighter bound in \eqref{eq:angle-bound}, and the relative error in estimating $\rhopt$. In a usual context, these quantities would not be available; therefore, they are also compared with two computable quantities, which are the spectral norm of the residual matrix, and the relative difference between two consecutive trace ratio values. 

The quantities depending solely on the trace ratio values seem to decay faster than $\sin ( V, \Vopt)$, meaning that a stopping criterion based on the former would give an unsatisfactory solution. The numerator in the bound \eqref{eq:angle-bound} approaches $\|R\|$ in a few iterations, suggesting that the norm of the residual is the dominant term. Both quantities follow $\sin ( V, \Vopt)$ with the same slope.
\end{example}

Motivated by the example, we further analyze the term $\rhopt - \rho$. This quantity can be mostly explained by $\|R\|^2$. Let $\beta^\ast = \tr((\Vopt)^T B \Vopt)$. It follows that
\begin{equation*}
(\rhopt - \rho)\,\beta^\ast = \tr((V^\ast)^TAV^\ast) - \rho\,\tr((V^\ast)^TBV^\ast) = \tr((V^\ast)^T(A-\rho\,B)\,V^\ast).    
\end{equation*}
From Courant--Fischer Theorem (see, e.g., \cite[Prop.~4.3]{ngo2012trace} for a similar application of the result) for the matrices $A-\rho\,B$ at the numerator and $B$ at the denominator, combined with the fact that $\sum_{i = 1}^k \lambda_i = 0$, one obtains
\[
0 \le\rhopt - \rho\le \frac{\sum_{i = 1}^k (\lambda_i(A-\rho B) - \lambda_i)}{\sum_{i=1}^k\lambda_{p-k+i}(B)}.
\]
Additionally, from the decomposition of $\Vopt$ as $\Vopt = V E + V_\perp F$,
\begin{align*}
(\rhopt - \rho)\,\beta^\ast
& = \tr\left((V^*)^T [V \ V_\perp][V \ V_\perp]^T(A-\rho B)[V \ V_\perp] [V \ V_\perp]^T V^* \right) = \tr\Big(
\smtxa{cc}{E^T & \!\! F^T}
\smtxa{cc}{\Lambda & \bar R^T \\ \bar R & \bar M}
\smtxa{c}{E \\ F}
\Big),
\end{align*} 
where $\bar R = V_\perp^T(A-\rho B) V$, $\bar M = V_\perp^T (A-\rho B) V_\perp$, and $\|\bar R\| = \| R\|$. 
Now let $\mu_i$ be the $i$th largest eigenvalue of the block diagonal matrix ${\rm diag}(\Lambda, \bar M)$. \cite[Cor.~1]{mathias1998quadratic} shows that $\vert \lambda_i(A-\rho B) - \mu_i \vert \le \delta^{-1}\,\|R\|^2$, where $\delta := \min_{i,j}\vert\lambda_i(\Lambda) - \lambda_j(\bar M)\vert$ measures the separation between the spectra of $\Lambda$ and $\bar M$. (This bound is enough for our purpose; a better and more recent result can be found in \cite{li2005note}). Then
\begin{align*}
\sum_{i = 1}^k \, (\lambda_i(A-\rho B) - \lambda_i) &\le \sum_{i = 1}^k \, (\,\vert\lambda_i - \mu_i \vert + \vert\lambda_i(A-\rho B) - \mu_i  \vert\,) \\
&\le \sum_{i = 1}^k \, \vert\lambda_i - \mu_i \vert + k\,\delta^{-1}\,\|R\|^2.
\end{align*}

This bound does not depend on $(\Vopt, \rhopt)$. The open question remains whether one can determine a bound for $\sum_{i = 1}^k \vert\mu_i - \lambda_i\vert$ in terms of $\|R\|$. Given the choices made for the expansion and the extraction phase, it is reasonable to expect that, as $\|R\|\to 0$, the eigenvalues of $\Lambda$ will coincide with the largest eigenvalues of ${\rm diag}(\Lambda, M)$. However, there might be some rare situations where $V$ fails to converge to the subspace spanned by the eigenvectors of $A-\rhopt B$, corresponding to the $k$ largest eigenvalues. The following example shows what can be observed in a typical case.
 
\begin{example}[Continuation of Example~\ref{ex:angle-bound}]
\rm The left plot of Figure~\ref{fig:eig-analysis} shows the behavior of $(\rhopt - \rho)\,\beta^\ast$. In this example, the squared norm of the residual seems to be an appropriate upper bound for this quantity. We also plot $\sum_{i = 1}^k \vert\mu_i - \lambda_i\vert$, which drops after a few iterations. At the same time, the eigenvalues of $H-\rho K$ approach the $k$ largest eigenvalues of $A-\rho B$. This can be seen in the right plot of Figure~\ref{fig:eig-analysis}.
\begin{figure}[htb!]
\centering
\includegraphics[width=\textwidth]{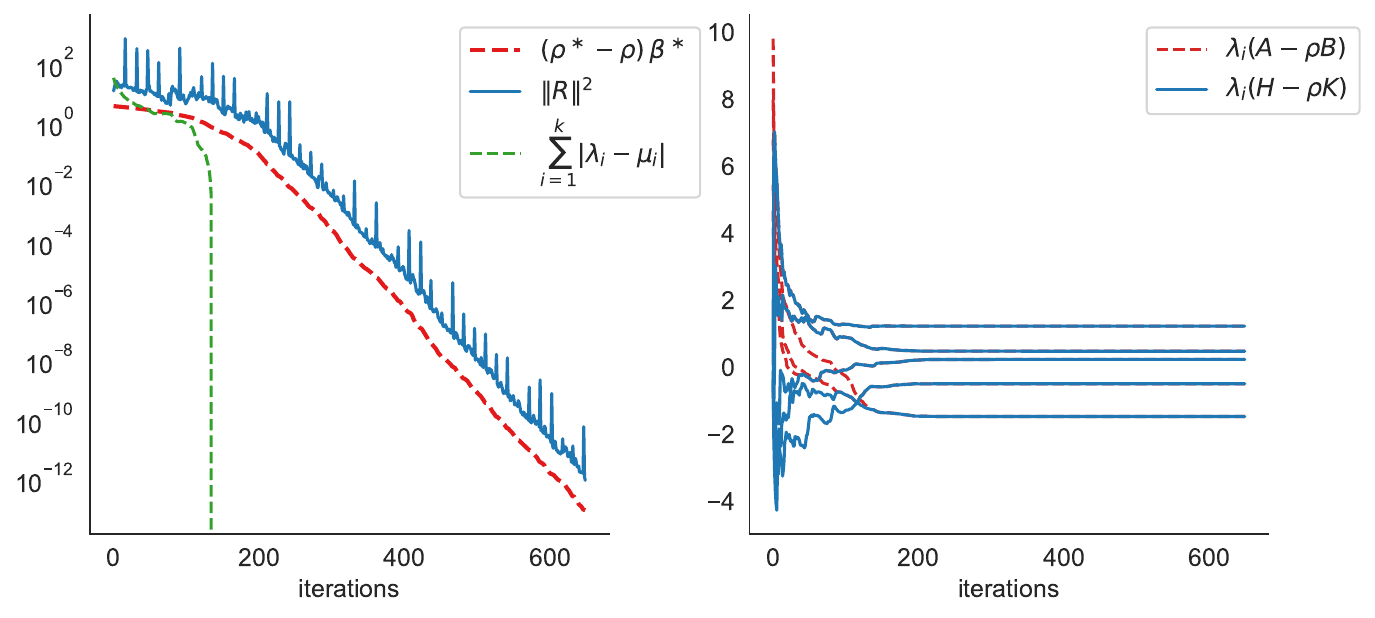}
\caption{On the left: $(\rho^\star - \rho)\beta^\ast$ is bounded by the squared norm of the residual matrix. On the right: the approximate eigenvalues $\lambda_i(H-\rho K)$ converge to the $k$ largest eigenvalues of $A-\rho B$ throughout the iterations.}
\label{fig:eig-analysis}
\end{figure}
\end{example}

\section{Classification task}\label{sec:classification}
Our experiments aim to compare the performance of Algorithm~\ref{algo:trsub-traceratio} with that of Algorithm~\ref{algo:traceratio-subspace}, and FDA, in the context of classification problems. 
We describe multigroup classification and introduce the relevant quantities involved in TR and FDA. 
In view of the numerical experiments of Section~\ref{sec:trsub-exp}, some practical aspects are discussed, including the implementation of a subspace method for FDA.

Suppose that a population is described by the random vector $\bX= (X_1, \dots,X_p)^T$, with mean $\expec(\bX) = \bmu$ and covariance $\var(\bX) = \Sigma$. It is common to assume that $\Sigma$ and its estimator are nonsingular.
The population is divided into $g$ groups (or classes), modeled by the class variable $Y \in \{1, \dots,g\}$. Denote the expectation of $\bX_i = (\bX\mid Y = i)$ by $\expec(\bX_i) = \bmu_i$, and its covariance matrix by $\var(\bX_i) = \Sigma_i$. The a priori probability of being in group $i$ is $p_i = P(Y=i)$. 

A new observation $\bx \in \mathbb{R}^{p}$ is assigned to one of the $g$ classes, via a so-called \emph{classification rule}, that is built on top of the available data. Let $\bx_{i1}, \dots, \bx_{in_i}$ represent $n_i$ realizations of a random sample of $\bX_i$; the total sample size is $n = \sum_{i=1}^g n_i$. Data are stored row-wise, in an $n\times p$ matrix $X$. The sample mean and the sample covariance matrix of group $i$ are $\bxbar_i = n_i^{-1} \, \sum_{h=1}^{n_i}{\bx_{ih}}$ and $S_i = n_i^{-1} \, \sum_{h=1}^{n_i} \, (\bx_{ih}-\bxbar_i) \, (\bx_{ih}-\bxbar_i)^T$, respectively.
The overall sample mean is estimated by $\bxbar = \sum_{i=1}^g \tfrac{n_i}{n} \, \bxbar_i$. 
The between-class scatter matrix, $S_B$, and the within-class scatter matrix, $S_W$ are defined as follows:
\begin{equation} \label{eq:SbSw}
S_B = n^{-1} \, \sum_{i=1}^g n_i \, (\bxbar_i-\bxbar) \, (\bxbar_i-\bxbar)^T, \qquad
S_W = n^{-1} \, \sum_{i=1}^g \sum_{h=1}^{n_i} \ (\bx_{ih}-\bxbar_i) \, (\bx_{ih}-\bxbar_i)^T.
\end{equation}
From the definition of the overall mean and the group means, $\rank(S_B) \le \min\{p, g-1\}$, while $\rank(S_W) \le \min\{p, n-g\}$. Both TR and FDA provide a subspace that minimizes the distance within the projected groups while maximizing the distance between the projected groups. Therefore the matrices $A$ and $B$ in \eqref{eq:TR.MaxProblem} for TR, and in \eqref{eq:gep} for FDA, are chosen as $S_B$ and $S_W$ respectively. 

In subspace methods, it is sufficient to know how the scatter matrices act on a vector, without storing them. The sum of the two scatter matrices gives the total sample covariance matrix:
\begin{equation}\label{eq:sumscatter}
S_T = n^{-1}\, \sum_{i=1}^g \sum_{h=1}^{n_i} \, (\bx_{ih}-\bxbar) \, (\bx_{ih}-\bxbar)^T = S_W + S_B. 
\end{equation}
Simple computations also show that $S_T = n^{-1} X^TX - \bxbar\,\bxbar^T$. In addition, the covariance matrices \eqref{eq:SbSw} can be written as the product of the following matrices (see, e.g., \cite{zhang2010fast}): 
\begin{align}
H_B &= [\sqrt{n_1}\,(\bxbar_1 - \bxbar),\,\dots,\,\sqrt{n_g}\,(\bxbar_g - \bxbar)] \in \mathbb R^{p\times g}\label{eq:deco}\\
H_W &= [\bx_{ih} - \bxbar_{i}\mid i=1,\dots,g\ \text{and}\ h = 1,\dots,n_i] \in \mathbb R^{p\times n}.\nonumber
\end{align}
Then $S_B = H_BH_B^T$ and $S_W = H_WH_W^T$. In practice, we use the decomposition of $S_B$, and compute $S_W = S_T - S_B$, which is more convenient. Then, apart from the data matrix, one would only need to store the group means and the overall mean to obtain the matrix-vector products $S_B\bv$ and $S_W\bv$, for $\bv\in \mathbb R^p$. The second product is the most expensive and amounts to $\calo(pn)$ operations; the first one requires $\calo(gp)$ flops. However, forming and storing $S_W$ costs $\calo(p^2n)$, and therefore it might still be more convenient to compute $H_W(H_W^T\bv)$ at each iteration. 


\subsection{A Davidson type method for FDA}
\label{sec:trsub-gep}
The trace ratio method is often compared to FDA in the context of multigroup classification. Therefore, we have implemented a Davidson type method to find the largest generalized eigenvalues of $(A,B)$ and their corresponding eigenvectors. This is described in Algorithm~\ref{algo:gep-davidson}.

Davidson type methods for the generalized eigenvalue problem can be found in the literature in, e.g., \cite{morgan1990davidson,morgan1991computing,romero2014parallel}. In \cite{morgan1991computing}, Davidson's method is used for the computation of harmonic Ritz values. As for the basis of the search subspace, there are two possibilities: choosing a basis with orthonormal or $B$-orthogonal columns. The extraction phase changes accordingly. The first case requires the solution to a projected generalized eigenvalue problem. In the second one, the $B$-orthogonal basis transforms the projected problem into a standard eigenvalue problem (see, e.g., \cite[Alg.~2]{romero2014parallel}). The first option is preferable, because it is generally favorable for the stability of the method, particularly for an ill-conditioned $B$. This choice is also suggested in \cite{morgan1990davidson,morgan1991computing} and \cite[Alg.~3]{romero2014parallel}. Note that, with a search space spanned by orthonormal vectors, the approximate eigenvectors are automatically $B$-orthogonal. 

\begin{algorithm}
\caption{A Davidson type method for FDA}
\label{algo:gep-davidson}
{\bf Input:} Symmetric $A \in \R^{p \times p}$, SPD $B \in \R^{p \times p}$, $k < p$,
maximum and minimum dimensions with $k \le \kmin < \kmax \le p$, block size $1\le\bsize\le k$, initial $p \times \kmin$ matrix $U_\kmin$ with orthogonal columns, tolerance {\sf tol}.\\
{\bf Output:} $p \times k$ matrix $V$ with orthonormal columns maximizing the trace ratio. \\
\begin{tabular}{ll}
{\footnotesize 1:} & Compute $AU_\kmin$, $BU_\kmin$ and $H_\kmin = U_\kmin^T\!AU_\kmin$, $K_\kmin = U_\kmin^T\!BU_\kmin$; $j = \kmin$\\
{\footnotesize 2:}& {\bf for} $i = 1, 2, \dots$ \\
{\footnotesize 3:} & \phantom{M} Solve GEP for $(H_j, K_j)$, store $k$ leading eigenvectors in $Z$, eigenvalues in $\Lambda$\\
{\footnotesize 4:} & \phantom{M} $V = U_jX$, \ $R=AV- BV\Lambda = AU_jZ- BU_jZ\Lambda$\\
{\footnotesize 5:} & \phantom{M} Compute $U_\bsize = [\bu_{j+1},\dots,\bu_{j+\bsize}]$ largest left singular vectors of $R$\\
{\footnotesize 6:} & \phantom{M} {\bf if} $\|R\| < {\sf tol}$,  \ {\bf return}, \ {\bf end} \\
{\footnotesize 7:} & \phantom{M} {\bf if}  $j = \kmax$:\\
{\footnotesize 8:} & \phantom{MM} Store $\kmin$ leading generalized eigenvectors of $(H_\kmax, K_\kmax)$ in $Z_\kmin = [Z\ Z_{\kmin-k}]$\\ 
{\footnotesize 9:} & \phantom{MM} Orthogonalize $Z_\kmin = Q_\kmin S_\kmin$\\
{\footnotesize 10:} & \phantom{MM} Set $j = \kmin$, shrink $U_j = U_\kmax Q_\kmin$,  $AU_j = AU_\kmax Q_\kmin$, $BU_j = BU_\kmax Q_\kmin$\\
 & \phantom{MMm} $H_j = Q_\kmin^T\!H_\kmax Q_\kmin$, $K_j = Q_\kmin^T\!K_\kmax Q_\kmin$ \\
{\footnotesize 11:} & \phantom{M} {\bf end}\\
{\footnotesize 12:} & \phantom{M} Reorthogonalize $U_\bsize$ against $U_j$, expand $U_{j+\bsize} = [U_j\ \ U_\bsize]$\\
{\footnotesize 13:} & \phantom{M} Compute $A\,U_\bsize$, $B\,U_\bsize$ and $U_{j+\bsize}^T\!AU_\bsize$, $U_{j+\bsize}^T\!BU_\bsize$ \\
{\footnotesize 14:} & \phantom{M} $j = j + \bsize$\\
{\footnotesize 15:} & {\bf end}
\end{tabular} 
\end{algorithm}

In the expansion phase, we maintain the choice made for TR and use the $1\le\bsize\le k$ largest left singular vectors of the residual matrix. Reorthogonalization is also performed, as in Section~\ref{sec:expa}.

The restarting procedure consists of using the ${\sf m_1}$ leading approximate generalized eigenvectors, extracted from the subspace $\calu_{\kmax}$. This is similar to what has been done in \cite{morgan1990davidson}, but also analogous to the restarting procedure in the Krylov--Schur method \cite{stewart2002krylov}.

Following a similar reasoning as in Section~\ref{sec:restart}, we show that it is unnecessary to perform the extraction phase immediately after restarting, to get a new residual matrix. First, the projected generalized eigenvalue problem for the first $k$ eigenvectors is also an optimization problem of the form \eqref{eq:gep}. As in Proposition~\ref{prop:mon-rho}, if $\calu_{\kmin}\subset\calu_{\kmax}$, it holds
\begin{equation}\label{eq:inclusion-gep}
\max_{V\,\subset\,\calu_{\kmin},\,V^TBV = I_k}\tr(V^T\!AV) \le \max_{V\,\subset\,\calu_{\kmax},\,V^TBV = I_k}\tr(V^T\!AV).
\end{equation}
Our restarting procedure attains equality in \eqref{eq:inclusion-gep}. Suppose the current search subspace is spanned by $U_\kmax\in \mathbb R^{p\times \kmax}$, and let $(H_\kmax, K_\kmax)$ be the corresponding projected $(A, B)$. Let $Z_\kmin = [Z\ Z_{\kmin - k}]$ contain the $\kmin$ largest generalized eigenvectors of $(H_\kmax, K_\kmax)$, where $Z$ is $\kmax\times k$, $Z_{\kmin-k}$ is $\kmax\times (\kmin - k)$; consider its QR decomposition $Z_\kmin = Q_\kmin S_\kmin = Q_\kmin[S_k \ S_{\kmin - k}]$, where $Q_\kmin$ is $\kmax\times\kmin$ with orthonormal columns, $S_\kmin$ is $\kmin\times\kmin$ and upper triangular, and $S_k$ is the first $\kmin\times k$ block of $S_\kmin$. 
The new search subspace $\calu_\kmin$ is defined as the span of $U_{\kmin} := U_\kmax Q_\kmin$, which has orthogonal columns. 

The extraction phase for $U_\kmin$ is 
\begin{align}
\label{eq:proj-gep}
\max_{V\,\subset\,\calu_{\kmin},\,V^TBV = I_k}\tr(V^T\!AV) = \max_{\substack{V = U_\kmin  C,\\C^TU_\kmin^TBU_\kmin C = I_k}}\tr(C^TU_{\kmin}^TAU_\kmin C),
\end{align}
where
\begin{align*}
I_k &= C^TU_\kmin^TBU_\kmin C = C^T Q_\kmin^T U_\kmax^T B U_\kmax Q_\kmin C \\ 
&= C^T S_\kmin^{-T} Z_\kmin^T K_\kmax Z_\kmin S_\kmin^{-1}C = C^TS_\kmin^{-T}S_\kmin^{-1}C,
\end{align*}
since $Q_\kmin = Z_\kmin S_\kmin^{-1}$, and $Z_\kmin$ has $K_\kmax$-orthogonal columns. In addition $Z_\kmin^TH_\kmax Z_\kmin = \Lambda_\kmin$, where $\Lambda_\kmin$ contains the first $\kmin$ generalized eigenvalues of $(H_\kmax,K_\kmax)$. Therefore $C^TU_{\kmin}^TAU_\kmin C = C^TS_\kmin^{-T}\Lambda_\kmin S_\kmin^{-1}C$. This establishes the equivalence between \eqref{eq:proj-gep} and
\[
\max_{C^TC = I_k}\tr(C^T\Lambda_\kmin C),
\]
whose maximizer is $C = [I_k\ 0]^T$. The maximum corresponds to the sum of the $k$ largest generalized eigenvalues of $(H_\kmax, K_\kmax)$. This means that $U_\kmin C$ achieves equality in \eqref{eq:inclusion-gep}. In addition, since $V = U_{\kmin}C = U_\kmax Q_\kmin S_k$, the residual matrix related to the new iteration can be computed by using the old quantities, without performing the extraction phase \eqref{eq:proj-gep}.

\subsection{Regularization in the large-scale setting}\label{sec:trsub-reg}
The within-scatter matrix may be numerically singular. In view of Proposition~\ref{prop:rank}, this might lead to an infinite trace ratio for TR, and to a so-called singular pencil for the GEP, which is known to be challenging to solve. One possible solution would be to cut out the nullspace of $S_W$ and solve either TR or FDA in the range of $S_W$. This has been done in, e.g., \cite{ngo2012trace}. 
However, in the large-scale setting, performing a full eigenvalue decomposition to find the range of $S_W$ is undesirable (and possibly unfeasible). Therefore we regularize $S_W$ and consider the TR (FDA) for $(S_B, (1-\alpha)S_W + \alpha I_p)$, with $\alpha\in (0, 1]$. This idea was originally proposed by \cite{friedman1989regularized} in the context of quadratic discriminant analysis. 

It is easy to show that this problem is equivalent to TR (FDA) for $(S_B, \,S_W + \frac{\alpha}{1-\alpha}I_p)$, with $\frac{\alpha}{1-\alpha} \in [0,\infty)$. The latter formulation justifies the omission of the multiplicative factor $\tr(S_W)/p$, which is present in \cite[Eq.~(18)]{friedman1989regularized}. 
There is no easy way of choosing the right amount of regularization. One may select $\alpha$ such that the subspace $V(\alpha)$ leads to the highest performance metric of interest in the classification task. 

A regularization of TR has already been proposed in \cite{zhang2010fast}, for high-dimensional problems, i.e., situations where the number of data points is much smaller than the dimension of the problem ($n \ll p$), and the data points are linearly independent. 
In this scenario, \cite[Thm.~4.2]{zhang2010fast} shows the equivalence between the solution to a regularized TR and a trace ratio method preceded by the QR decomposition of the data matrix $X$. In the high-dimensional setting, a QR step should be quite cheap, given that $n \ll p$, and leads to a smaller TR problem. 
In this case, QR combined with TR would be a better strategy than exploiting subspace methods. 

\subsection{Classification rule}\label{sec:rule} 
TR and FDA are linear dimensionality reduction methods, i.e., they provide a subspace of a given dimension, which should maximize the separation between the groups in the data matrix. The quality of the solution is usually assessed on the capability of the model to correctly classify new observations, given the projected data. 

In the experiments, TR and FDA solutions are evaluated in terms of classification accuracy, which is the proportion of data points that have been correctly classified. 
Here the classification rule is linear discriminant analysis (LDA, see, e.g., \cite[Sec.~4.3]{Hastie2009StatLearning}). A new observation is first projected onto the subspace spanned by either the solution to TR or FDA, indicated by $V$; then it is classified to the group with the closest projected centroid, based on the Mahalanobis distance and the number of samples in each group. In other words, if $\bx$ is a new observation, then it is assigned to group $j$ if
\[
j = \argmin_{i}\|V^T(\bx-\bxbar_i)\|^2_{(V^TS_{\rm pooled}V)^{-1}} - 2\log\tfrac{n_i}{n},
\]
where $V^TS_{\rm pooled}V = \frac{n}{n-g}V^TS_WV$. Following the linear dimensionality reduction step, other classifiers may be chosen for assigning the new observations. However, a deeper study on this matter is out of the scope of the paper.

\section{Experiments}\label{sec:trsub-exp}
We compare the two versions of TR (cf.~Algorithm~\ref{algo:trsub-traceratio} and Algorithm~\ref{algo:traceratio-subspace}) and the subspace method for FDA (cf. Algorithm~\ref{algo:gep-davidson}). The methods are evaluated in terms of number of matrix-vector (MV) products and computational time (in seconds). The quality of their solution is assessed in terms of classification accuracy, by implementing the classification rule in Section~\ref{sec:rule}. 
As Algorithm~\ref{algo:trsub-traceratio} and Algorithm~\ref{algo:traceratio-subspace} yield approximately the same solution, the comparison of their accuracy acts mainly as a sanity check. 

Both synthetic data and real datasets are considered, with $n\gg p$, for the reasons remarked at the end of Section~\ref{sec:trsub-reg}. Unless otherwise specified, the block size is $\bsize = 1$.

\subsection{Synthetic data}

The synthetic example is taken from \cite{ortner2020robust}, with $g$ groups, $g$ relevant features, and $q > 0$ other irrelevant variables. 
The input variables of the $i$th class are $\bX_i \sim \mathcal{N}_{g+q}(2\be_i, \Sigma)$, for $i = 1,\dots,g$, where $\be_i$ is the $i$th vector of the canonical basis of $\mathbb R^{g+q}$, and the covariance matrix is
\[
\Sigma = 
 \begin{bmatrix}
 1   & 0.1 & \cdots & 0.1 &\\
 0.1 & 1   & 0.1 & \vdots & \\
 0.1 & 0.1 & \ddots   & 0.1 &\\
 0.1 & \cdots & 0.1 &   1 &\\
     &     &     &     & I_q
 \end{bmatrix}
\]
for all groups. In other words, $\Sigma$ is a block diagonal matrix, where the first $g\times g$ block has ones on the main diagonal, and 0.1 as off-diagonal elements. The second block is a $q\times q$ identity matrix. The remaining elements are zeros. In our experiments, there are $g = 3$ groups, $q = 5000$ irrelevant variables, $n_i = 50000$ training data points per group, and $n_i = 1000$ test data points per group. We run $50$ random experiments, where $(S_B, S_W)$ are estimated according to \eqref{eq:SbSw}, using the training data. In both TR and FDA, data are projected onto a subspace of dimension $k = 2$. TR in Algorithm~\ref{algo:trsub-traceratio} is indicated as \trks{}, given that the eigenvalues in the inner iterations are computed with the Krylov--Schur method. 
Davidson's methods for TR and FDA are indicated as \trsub{} and \fdasub{}, respectively. In all methods, the algorithms stop when $10^5$ (outer) iterations are reached. The stopping criterion is based on the spectral norm of the residual (see Section~\ref{sec:expa}): $\|R\|<{\sf tol}$, with ${\sf tol} = 10^{-6}$. For the Krylov--Schur method, \trsub{} and \fdasub{}, the minimum size of the search subspace is $\kmin = 2k = 4$, and the maximum size is $\kmax = 4k = 8$.

In this first experiment, we also test two different ways of performing the matrix-vector multiplication by the within covariance matrix. While the product with $S_B$ is always computed via the decomposition \eqref{eq:deco}, in one case, the within-scatter matrix $S_W$ \eqref{eq:SbSw} is precomputed and stored. Alternatively, the relations \eqref{eq:sumscatter} and \eqref{eq:deco} are exploited to decompose both $S_B$ and $S_W$, with very little extra storage. For a fair comparison of the two techniques, the computational time for precomputing the scatter matrices is also taken into account.

Table~\ref{tab:ort-clean} shows the result of the simulations. All methods converge within $10^5$ iterations in all experiments. \trsub{} and \trks{} converge to the same trace ratio value $\rhopt$, meaning that \trsub{} also reaches the global optimum of \eqref{eq:TR.MaxProblem}. This behavior is observed in the following experiments as well (cf.~Tables~\ref{tab:fashion} and \ref{tab:german}). Indeed, as pointed out in Section~\ref{sec:expa}, it is unlikely that \trsub{} fails to converge to the desired solution. 

In this example, all methods give the same accuracy, even if, in general, TR and FDA provide different solutions. In particular, while the projection matrix of TR has orthonormal columns, the columns of the FDA solution are $S_W$-orthogonal. 

\fdasub{} and \trsub{} show very similar performances, in terms of both average matrix-vector products and average computational time. \trks{} requires more matrix-vector products and computational time. Because of this, the method benefits from the precomputation of the within covariance matrix. On the other hand, when fewer matrix-vector products are required (i.e., in \trsub{} and \fdasub{}), it is more convenient to compute the matrix-vector products by decomposing the scatter matrices.

\begin{table}
\centering
\caption{
Results for the synthetic data, over $50$ experiments. Average (Avg) and standard deviation (Sd) of matrix-vector (MV) products, time (in seconds, $s$), and accuracy are reported. For TR, we record the final trace ratio value $\rhopt$ and the corresponding eigengap $\lambda_k(A-\rho^\ast B) - \lambda_{k+1}(A-\rho^\ast B)$. It is also indicated whether the covariance matrices were precomputed and stored, or not.}
\vspace{1mm}
\footnotesize
\begin{tabular}{lclcccccccccc}
\toprule
Precomputed & $k$ & Method & \multicolumn{2}{c}{MV} & \multicolumn{2}{c}{Time ($s$)} & \multicolumn{2}{c}{Accuracy} & \multicolumn{2}{c}{$\rho^\ast$} & \multicolumn{2}{c}{Eigengap} \\
 &  &  & Avg & Sd & Avg & Sd & Avg & Sd & Avg & Sd & Avg & Sd \\
\midrule
No  & 2 & FDA subspace  & 23 & 0 & \ph{1}9.4  & 0.1 & 0.85 & $< 0.01$ &                                       \\
No  & 2 & TR KSchur     & 70 & 0 & 27.1 & 0.4 & 0.85 & $< 0.01$ & 1.42 & $< 0.01$ & 0.94 & $< 0.01$     \\ \vspace{2mm}
No  & 2 & TR subspace   & 25 & 0 & 10.0 & 0.1 & 0.85 & $< 0.01$ & 1.42 & $< 0.01$ & 0.94 & $< 0.01$     \\
Yes & 2 & FDA subspace  & 23 & 0 & 14.9 & 0.2 & 0.85 & $< 0.01$ &                                       \\
Yes & 2 & TR KSchur     & 70 & 0 & 15.2 & 0.2 & 0.85 & $< 0.01$ & 1.42 & $< 0.01$ & 0.94 & $< 0.01$     \\
Yes & 2 & TR subspace   & 25 & 0 & 14.9 & 0.2 & 0.85 & $< 0.01$ & 1.42 & $< 0.01$ & 0.94 & $< 0.01$     \\
\bottomrule
\end{tabular}
\label{tab:ort-clean}
\end{table}

\subsection{Real datasets} 
In what follows, we consider two real datasets to highlight problem-dependent behaviors of \trsub{}, \trks{}, and \fdasub{}, for different levels of regularization $\alpha$ and reduced dimension $k$.

\paragraph{Fashion MNIST \cite{fashionmnist}} The dataset consists of $n = 60000$ gray-scale images of fashion articles, with $p = 784$ pixels. The pixels have been scaled by $255$, so each feature ranges in $[0,1]$. The number of classes is $g = 10$. All classes are equally represented. Data are projected onto a subspace of size $k = 9$ and then new items are classified with the LDA classifier. The classification accuracy is assessed via a 10-fold cross-validation. 
For all the methods \trsub{}, \trks{}, and \fdasub{}, the size of the search subspace varies between $\kmin = 2k = 18$ and $\kmax = 5k = 45$. Here we show how regularization affects the various methods, by setting $\alpha = 0$ (i.e., no regularization) and $\alpha = 0.1$. Results are shown in Table~\ref{tab:fashion}.

\begin{table}[ht!]
\centering
\caption{
Results of the 10-fold cross-validation on {\sf Fashion MNIST}. Average (Avg) and standard deviation (Sd) of matrix-vector (MV) products, time (in seconds, $s$), and accuracy are reported. For TR, we record the final trace ratio value $\rhopt$ and the corresponding eigengap $\lambda_k(A-\rho^\ast B) - \lambda_{k+1}(A-\rho^\ast B)$. Different levels of regularization $\alpha$ are considered.}
\vspace{1mm}
\footnotesize
\begin{tabular}{lclccccccccccc}
\toprule
$\alpha$ & $k$ & Method & \multicolumn{2}{c}{MV} & \multicolumn{2}{c}{Time ($s$)} & \multicolumn{2}{c}{Accuracy} & \multicolumn{2}{c}{$\rho^\ast$} & \multicolumn{2}{c}{Eigengap} \\
 &  &  & Avg & Sd & Avg & Sd & Avg & Sd & Avg & Sd & Avg & Sd\\
\midrule
0 & 9 & FDA subspace & 16675 & 1124 & 9.2 & 0.7 & 0.82 & $< 0.01$ \\
0 & 9 & TR KSchur & 79031 & 4064 & 8.4 & 0.7 & 0.48 & \ph{$<$}0.01          & 13.51 & \ph{$<$}0.03     & $4.4 \cdot 10^{-5}$       & $1.4 \cdot 10^{-5}$     \\ \vspace{2mm}
0 & 9 & TR subspace & \ph{1}9124 & \ph{1}287 & 6.2 & 0.7 & 0.48 & \ph{$<$}0.01          & 13.51 & \ph{$<$}0.03     & $4.4 \cdot 10^{-5}$       & $1.4 \cdot 10^{-5}$     \\
0.1 & 9 & FDA subspace & \ph{11}163 & \ph{111}1 & 0.6 & 0.0 & 0.80 & $< 0.01$       \\
0.1 & 9 & TR KSchur & 34526 & 1527 & 4.0 & 0.5 & 0.59 & $< 0.01$    & \ph{1}5.14  & $< 0.01$ & $7.4 \cdot 10^{-6}$ & $2.9 \cdot 10^{-6}$     \\
0.1 & 9 & TR subspace & \ph{1}6917 & \ph{1}186 & 4.4 & 0.2 & 0.59 & $< 0.01$    & \ph{1}5.14  & $< 0.01$ & $7.4 \cdot 10^{-6}$ & $2.9 \cdot 10^{-6}$     \\
\bottomrule
\end{tabular}
\label{tab:fashion}
\end{table}

The method \trsub{} requires fewer matrix-vector products compared to \trks{}, but it is less competitive in terms of computational time when $\alpha = 0.1$. All three algorithms benefit from the regularization of $S_W$. This holds especially for \fdasub{}. Regularization also improves the accuracy of the TR problem, from $48\%$ to $59\%$. 

We now compare the behavior of \trks{}, \trsub{}, and \fdasub{} throughout the iterations. We consider one of the splits in the 10-fold cross-validation, and plot the number of matrix-vector products per iteration against the spectral norm of the residual matrix, which constitutes the stopping criterion. The two subspace methods are also run without restart, which corresponds to setting $\kmax = p = 784$. The idea is to check how well the restarting procedure behaves. 

Results are shown in Figure~\ref{fig:f-mnist}. As already suggested by Table~\ref{tab:fashion}, the behavior of FDA changes significantly when regularization is added, for both $\kmax = 5k$ and $\kmax = p$. The method does not show as many oscillations in $\|R\|$ as in the case $\alpha = 0$. A possible hint for the slow convergence might be found in the fact that, before regularization, $\lambda_p(S_W)\approx 10^{-7}$. When $\alpha = 0.1$, $\lambda_p((1-\alpha)S_W + \alpha I_p) \approx 10^{-2}$. An advantage of this regularization is that an $S_W$ with a larger minimal singular value avoids the problematic occurrence of almost singular pencils.
Regularization seems to have a smaller impact on the behavior of \trsub{} and \trks{}. In particular, \trsub{} with restart shows many oscillations in both cases.
\trks{} behaves differently since it performs fewer outer iterations, that require several inner matrix-vector products. As already observed, it converges in more matrix-vector products than \trsub{}.

Figure~\ref{fig:f-mnist-rho} shows the value of the trace ratio per number of matrix-vector products. Interestingly, \trsub{} gets close to the optimal trace ratio value faster than \trks{}. 

\begin{figure}[htb!]
\centering
\includegraphics[width=\textwidth]{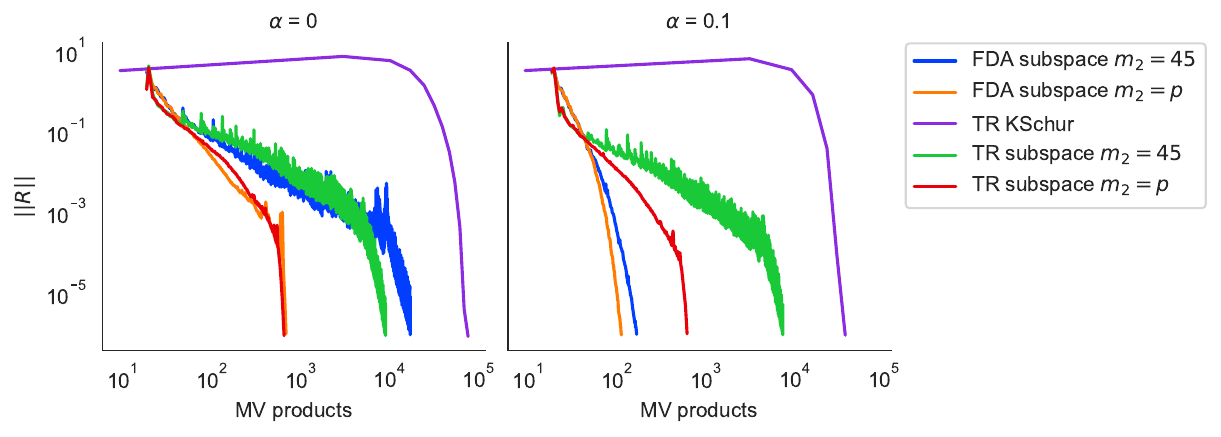}
\caption{FDA and TR for Fashion MNIST, solved by various methods. The number of matrix-vector (MV) products is plotted against the spectral norm of the residual matrix, $\|R\|$. The graphs are for different levels of regularization in $S_W$.}
\label{fig:f-mnist}
\end{figure}

\begin{figure}[htb!]
\centering
\includegraphics[width=\textwidth]{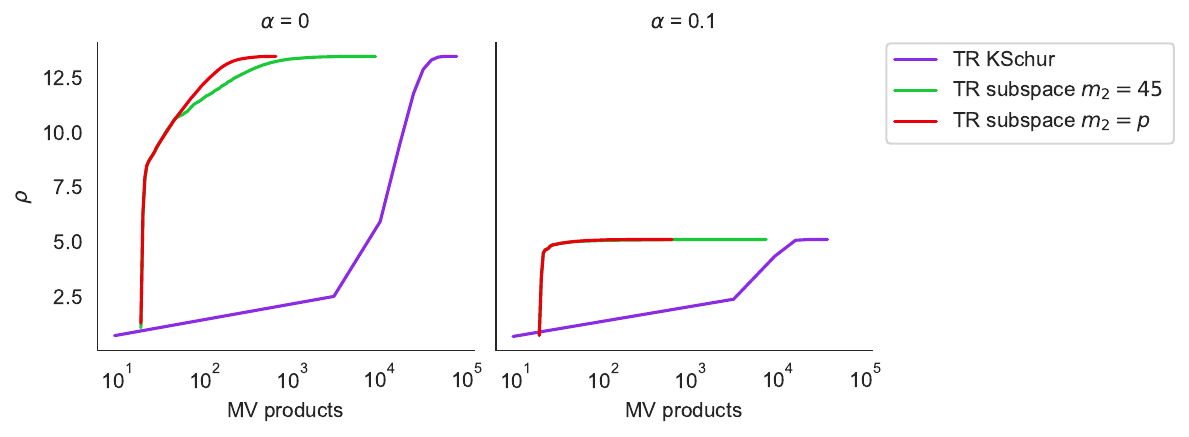}
\caption{Trace ratio value $\rho$ per number of matrix-vector (MV) products. The graphs are for different levels of regularization in $S_W$.}
\label{fig:f-mnist-rho}
\end{figure}

\paragraph{German Traffic Sign Recognition \cite{german2011}} The second dataset consists of $n = 39209$ traffic signs, described by $p = 6052$ features (HOG features, see \cite{german2011} for more details), and divided into $g = 43$ classes. The dataset is unbalanced, with $n_i\in[270, 3000]$. 
The classification accuracy is assessed via a 10-fold cross-validation, with the LDA classifier. 
For all the methods \trsub{}, \trks{}, and \fdasub{}, the size of the search subspace is between $\kmin = 2k$ and $\kmax = 5k$.  The large number of classes allows us to consider several reduced dimensions $k\in\{10,20,30\}$. Block versions of the methods are also tested, with block size $\bsize = 5$ (cf.~\cite{zhou2008block} for a block Krylov--Schur method). 
In this example, the regularization parameter is set to $\alpha = 0.1$. 

The results of the 10-fold cross-validation are shown in Table~\ref{tab:german}. For all $k$, FDA gives slightly better results in terms of accuracy. Increasing $k$ leads to an increase in the average number of matrix-vector products (and computational time) for all methods. \trsub{} and \fdasub{} methods with block size $\bsize = 5$ are on average faster than their counterparts with $\bsize = 1$ in terms of computational time; nevertheless \trsub{} with $\bsize = 5$ tends to require more matrix-vector products than \trsub{} with $\bsize = 1$. \trks{} with $\bsize = 5$ seems to be slower than \trks{} with $\bsize = 1$, and shows a larger variability compared to the other methods. Among the TR methods, \trsub{} with $\bsize=1$ is the one that requires the smallest average number of matrix-vector products for convergence; \trsub{} with $\bsize = 5$ is the least expensive in terms of computational time.

\begin{table}[ht!]
\centering
\caption{
Results of the 10-fold cross-validation on {\sf German Traffic Sign Recognition}. Average (Avg) and standard deviation (Sd) of matrix-vector (MV) products, time (in seconds, $s$), and accuracy are reported. For TR, we record the final trace ratio value $\rhopt$ and the corresponding eigengap $\lambda_k(A-\rho^\ast B) - \lambda_{k+1}(A-\rho^\ast B)$. Different reduced dimensions $k$ are considered. The regularization parameter is set to $\alpha = 0.1$.} 
\vspace{1mm}
\footnotesize
\begin{tabular}{rlrrrrrcrccc}
\toprule
$k$ & Method & \multicolumn{2}{c}{MV} & \multicolumn{2}{c}{Time ($s$)} & \multicolumn{2}{c}{Accuracy} & \multicolumn{2}{c}{$\rho^\ast$}& \multicolumn{2}{c}{Eigengap} \\
  &  & Avg & Sd & Avg & Sd & Avg & Sd & Avg & Sd & Avg & Sd \\
\midrule
10 & FDA subspace               & 224   & 1     & 14.2  & 0.5   & 0.83 & \ph{$<$}0.01  \\
10 & FDA subspace, $\bsize = 5$ & 230   & 2     & 12.1  & 0.5   & 0.83 & \ph{$<$}0.01  \\
10 & TR KSchur                  & 1016  & 13    & 21.7  & 0.4   & 0.74 & \ph{$<$}0.01     & 9.48 & \ph{$<$}0.01     & $6.1 \cdot 10^{-3}$ & $4.9 \cdot 10^{-4}$ \\
10 & TR KSchur, $\bsize = 5$    & 7612  & 5206  & 30.9  & 13.3  & 0.74 & \ph{$<$}0.01     & 9.48 & \ph{$<$}0.01     & $6.1 \cdot 10^{-3}$ & $4.9 \cdot 10^{-4}$ \\
10 & TR subspace                & 572   &  8    & 19.1  & 0.5   & 0.74 & \ph{$<$}0.01     & 9.48 & \ph{$<$}0.01     & $6.1 \cdot 10^{-3}$ & $4.9 \cdot 10^{-4}$ \\ \vspace{2mm}
10 & TR subspace, $\bsize = 5$  & 1169  & 73    & 15.7  & 0.6   & 0.74 & \ph{$<$}0.01     & 9.48 & \ph{$<$}0.01     & $6.1 \cdot 10^{-3}$ & $4.9 \cdot 10^{-4}$ \\

20 & FDA subspace               & 338   & 2     & 17.0  & 0.5   & 0.93 & \ph{$<$}0.01  \\
20 & FDA subspace, $\bsize = 5$ & 340   & 0     & 12.7  & 0.5   & 0.93 & \ph{$<$}0.01  \\
20 & TR KSchur                  & 2524  & 99    & 40.0  & 1.5   & 0.92 & $< 0.01$ & 6.12 & $< 0.01$ & $5.8 \cdot 10^{-4}$ & $6.8 \cdot 10^{-5}$ \\
20 & TR KSchur, $\bsize = 5$    & 9306  & 1711  & 37.0  & 5.3   & 0.92 & $< 0.01$ & 6.12 & $< 0.01$ & $5.8 \cdot 10^{-4}$ & $6.8 \cdot 10^{-5}$ \\
20 & TR subspace                & 1704  & 34    & 43.0  & 1.1   & 0.92 & $< 0.01$ & 6.12 & $< 0.01$ & $5.8 \cdot 10^{-4}$ & $6.8 \cdot 10^{-5}$ \\ \vspace{2mm}
20 & TR subspace, $\bsize = 5$  & 2990  & 237   & 24.6  & 1.1   & 0.92 & $< 0.01$ & 6.12 & $< 0.01$ & $5.8 \cdot 10^{-4}$ & $6.8 \cdot 10^{-5}$ \\

30 & FDA subspace               & 419   & 0     & 20.5  & 0.7   & 0.95 & $< 0.01$  \\
30 & FDA subspace, $\bsize = 5$ & 420   & 2     & 13.5  & 0.5   & 0.95 & $< 0.01$  \\
30 & TR KSchur                  & 6936  & 360   & 97.5  & 4.8   & 0.94 & $< 0.01$ & 4.63 & $< 0.01$ & $3.6 \cdot 10^{-5}$ & $4.3 \cdot 10^{-6}$ \\
30 & TR KSchur, $\bsize = 5$    & 35982 & 4405  & 114.2 & 13.2  & 0.94 & $< 0.01$ & 4.63 & $< 0.01$ & $3.6 \cdot 10^{-5}$ & $4.3 \cdot 10^{-6}$ \\
30 & TR subspace                & 4511  & 101   & 124.6 & 4.8   & 0.94 & $< 0.01$ & 4.63 & $< 0.01$ & $3.6 \cdot 10^{-5}$ & $4.3 \cdot 10^{-6}$ \\ 
30 & TR subspace, $\bsize = 5$  & 11176 & 1377  & 75.2  & 8.6   & 0.94 & $< 0.01$ & 4.63 & $< 0.01$ & $3.6 \cdot 10^{-5}$ & $4.3 \cdot 10^{-6}$ \\
\bottomrule
\end{tabular}
\label{tab:german}
\end{table}

In this problem, the slow convergence of TR is probably due to the small gap between the $k$th and the $(k+1)$st largest eigenvalues of the matrix $S_B - \rhopt S_W$ (cf.~the discussion in Section~\ref{sec:tr-overview}). 
Indeed, Table~\ref{tab:german} shows that the gap decreases as the reduced dimension increases.

When the gap is numerically zero and uniqueness is lost, there is no guarantee that different solutions will provide the same classification results. Therefore, one might consider monitoring $\lambda_k(U^T(A-\rho B)U) - \lambda_{k+1}(U^T(A-\rho B)U)$ as an approximate value for $\lambda_k(A-\rhopt B) - \lambda_{k+1}(A-\rhopt B)$, where $U$ is the output search subspace of Algorithm~\ref{algo:traceratio-subspace}. If this gap is not satisfactory, the user may choose a different reduced dimension to run TR.

Among the compared methods, \fdasub{} tends to be the most economic, as it considers a fixed GEP. It does have to solve small GEPs in each step, which, depending on the context, may be relatively expensive.
In contrast, TR solves (cheaper) standard eigenvalue problems, but with the trace ratio value changing over the iterations. 

The main cost in \trks{} lies in matrix-vector products, followed by the corresponding re-orthogonalization steps.
While \trsub{} usually requires fewer MVs, it involves relatively expensive projected TR problems, the computation of the residual matrix and its dominant singular value(s) and vector(s); that is why it is sometimes slower than \trks{} in terms of computational time. All methods will generally be computationally favorable for large-scale problems, since the projected problems and the other operations would be relatively inexpensive compared to MVs.

Finally, TR leads to a different low-dimensional space for the classification than FDA. This may correspond to better classification results for TR in some cases (cf., e.g., \cite{ngo2012trace}), although this was not observed in our experiments.

\section{Conclusions} \label{sec:trsub-concl}
We have introduced a Davidson type subspace method for large-scale trace ratio problems. The details about the algorithm are covered in Section~\ref{sec:tr-davidson}. In the extraction phase, a projected trace ratio problem is solved through the conventional Newton type routine for TR, as outlined in Algorithm~\ref{algo:trsub-traceratio}. The information gathered within the residual matrix has been used to expand the search subspace. The restart strategy is designed to maintain a non-decreasing trace ratio value throughout the iterations. 

Section~\ref{sec:trsub-analysis} analyzes the behavior of the approximate solution to TR. As the angle between the search subspace and the solution to TR approaches zero, there is a subset of approximate eigenvalues approaching the largest eigenvalues of $A - \rhopt B$. This stems from the trace ratio value converging towards $\rhopt$ as the accuracy of the search subspace improves. If the largest eigenvalues of $A - \rhopt B$ are uniformly separated from some unwanted approximate eigenvalues, the angle between the approximate solution and the solution to TR also approaches zero. 

Additionally, we have shown that this angle is bounded by the spectral norm of the residual matrix, augmented by a term that depends on the gap between $\rhopt$ and the current approximate trace ratio value. In typical situations, we expect this quantity to be bounded quadratically by the norm of the residual matrix. To some extent, this fact justifies the use of the spectral norm of the residual matrix in the computation of the stopping criterion.

The trace ratio problem finds common application in multigroup classification, revisited in Section~\ref{sec:classification} along with an implementation of a Davidson type method for FDA. 

Section~\ref{sec:trsub-exp} presents a numerical analysis of subspace methods for both the trace ratio problem (\trsub{}) and FDA (\fdasub{}), along with the Newton-type algorithm for TR (\trks{}). This comparison involves both synthetic and real datasets. It seems that \trsub{} requires fewer matrix-vector products than \trks{} to converge to the TR maximizer. The difference is even more pronounced in the convergence of the trace ratio value towards $\rhopt$. In the context of real datasets, and in contrast to \fdasub{}, the norm of the residual matrix exhibits more oscillations in \trsub{}. This may be partly caused by a relatively narrow gap between the $k$th and $(k+1)$st eigenvalues of $A-\rhopt B$. While the regularization technique discussed in Section~\ref{sec:trsub-reg} has proven effective in improving the convergence speed of FDA, finding an appropriate regularization strategy for TR might be challenging, as the width of the gap itself depends on the optimal trace ratio value.

An implementation of \trks{}, \trsub{}, and \fdasub{} is available in Matlab at \href{https://github.com/gferrandi/traceratio-matlab}{github.com/gferrandi/traceratio-matlab}.

{\bf Acknowledgments:}  
This work has been supported by the European Unions Horizon 2020 program under the Marie Sklodowska-Curie Grant Agreement no.~812912.
It has also received support from Fundação para a Ciência e Tecnologia, 
Portugal, through the project UIDB/04621/2020, with DOI \href{https://doi.org/10.54499/UIDB/04621/2020}{10.54499/UIDB/04621/2020}. 

\bibliographystyle{plain}
\bibliography{tidy-references}

\end{document}